\documentclass[reqno]{amsart}
\usepackage{amsaddr}
\usepackage{amsmath,amssymb,amsfonts,amsthm}
\usepackage{mathtools}
\usepackage{graphicx}
\usepackage{subcaption}
\usepackage{mathptmx}
\usepackage{color}
\usepackage{xspace}
\usepackage{array}
\usepackage{pgfplots}
\usepackage{ifthen}
\mathtoolsset{centercolon}
\usepackage{algorithm}
\usepackage{algorithmic}
\usepackage{cite}
\usepackage[colorlinks=true,pagebackref=true,citecolor = cyan,linkcolor=cyan,hypertexnames=false]{hyperref}
\renewcommand*{\backref}[1]{}
\renewcommand*{\backrefalt}[4]{[{\footnotesize%
		\ifcase #1 Not cited.%
		\or Cited on page~#2.%
		\else Cited on pages #2.%
		\fi%
	}]}
\usepackage{cleveref}
\usepackage{autonum}
\usepackage{tikz}
\usepackage{pgfplotstable}
\usetikzlibrary{arrows,positioning,plotmarks,external,patterns,angles,
decorations.pathmorphing,backgrounds,fit,shapes,calc}

\newcommand{\N}{\mathbb{N}}
\newcommand{\Z}{\mathbb{Z}}
\newcommand{\R}{\mathbb{R}}

\newcommand{\C}{\mathbb{C}}
\newcommand{\fgrad}{\nabla f}
\newcommand{\lyp}{\mathcal{E}}
\newcommand{\pro}[2]{\left\langle{#1},{#2}\right\rangle}
\newcommand{\dif}[2]{\frac{{\mathrm d}{#1}}{{\mathrm d}{#2}}}

\newcommand{\Order}[1]{\mathrm{O} \left( #1 \right) }
\newcommand{\rd}{\mathrm{d}}

\def\trans{\text{\tiny\sf T}}

\let\Re\relax 
\DeclareMathOperator{\Re}{Re}

\DeclarePairedDelimiter\paren{\lparen}{\rparen}
\DeclarePairedDelimiterX{\inpr}[2]{\langle}{\rangle}{{#1},{#2}}
\DeclarePairedDelimiterX{\L2ip}[2]{\langle}{\rangle_{L^2}}{{#1},{#2}}
\DeclarePairedDelimiterX{\setI}[2]{\{}{\}}{\,{#1}\ \delimsize| \ {#2}\,}
\DeclarePairedDelimiter{\setE}{\{}{\}}
\DeclarePairedDelimiter{\abs}{|}{|}
\DeclarePairedDelimiter{\norm}{\|}{\|}

\definecolor{mygreen}{rgb}{0,0.5,0}
\definecolor{myred}{rgb}{0.8,0,0}
\definecolor{myblue}{rgb}{0,0,0.5}

\newtheorem{theorem}{Theorem}[section]
\newtheorem{proposition}{Proposition}[section]
\newtheorem{lemma}[proposition]{Lemma}

\newtheorem{definition}{Definition}[section]
\newtheorem{remark}{Remark}[section]

\crefname{theorem}{Theorem}{Theorems}
\crefname{lemma}{Lemma}{Lemmas}
\crefname{proposition}{Proposition}{Propositions}
\crefname{corollary}{Corollary}{Corollaries}
\crefname{definition}{Definition}{Definitions}
\crefname{example}{Example}{Examples}
\crefname{remark}{Remark}{Remarks}
\crefname{figure}{Figure}{Figures}
\crefname{table}{Table}{Tables}

\crefname{section}{Section}{Sections}
\crefname{subsection}{Section}{Sections}

\title[Nesterov acceleration via linear multistep methods]{A novel interpretation of Nesterov's acceleration\\ via variable step-size linear multistep methods}

\begin{document}

\author{Ryota Nozawa}
\address{Department of Mathematical Informatics,
        Graduate School of Information Science and Technology,
        The University of Tokyo, Tokyo, Japan.}

\author{Shun Sato}
\address{Department of Mathematical Informatics,
        Graduate School of Information Science and Technology,
        The University of Tokyo, Tokyo, Japan.}
\email{shun@mist.i.u-tokyo.ac.jp}

\author{Takayasu Matsuo}
\address{Department of Mathematical Informatics,
        Graduate School of Information Science and Technology,
        The University of Tokyo, Tokyo, Japan.}

\begin{abstract}
Nesterov's acceleration in continuous optimization can be understood in a novel way when Nesterov's accelerated gradient (NAG) method is considered as a linear multistep (LM) method for gradient flow. 
Although the NAG method for strongly convex functions (NAG-sc) has been fully discussed, the NAG method for $L$-smooth convex functions (NAG-c) has not.
To fill this gap, we show that the existing NAG-c method can be interpreted as a variable step size LM (VLM) for the gradient flow. 
Surprisingly, the VLM allows linearly increasing step sizes, which explains the acceleration in the convex case. 
Here, we introduce a novel technique for analyzing the absolute stability of VLMs. 
Subsequently, we prove that NAG-c is optimal in a certain natural class of VLMs. 
Finally, we construct a new broader class of VLMs by optimizing the parameters in the VLM for ill-conditioned problems. 
According to numerical experiments, the proposed method outperforms the NAG-c method in ill-conditioned cases.
These results imply that the numerical analysis perspective of the NAG is a promising working environment, and considering a broader class of VLMs could further reveal novel methods.
\end{abstract}

\maketitle

\section{Introduction}

In this study, we use the gradient flow
\begin{equation}\label{eq:gf}
    \dot{x}(t)=-\nabla f(x)\qquad(x(0)=x_0),
\end{equation}
and consider the unconstrained optimization problem
\begin{equation}{\label{main problem}}    
    \min_{x\in \mathbb{R}^d} \; f(x),
\end{equation}
where $f \colon \R^d \to \R $ is $L$-smooth ($f$ is differentiable and the gradient $ \nabla f $ is $L$-Lipschitz continuous) and convex. Let $x^* \in \R^d $ denote an optimal solution of \eqref{main problem}.
The gradient descent method $ x_{n+1} = x_n - h_n \nabla f ( x_n ) $ for solving the problem can be regarded as an explicit Euler method for the gradient flow~\eqref{eq:gf}. 
The convergence rate of the gradient descent method is 
$
    f(x_n)-f(x^*)=\Order{\frac{1}{n}}
$
(cf.~\cite{N2004}),
whereas that of the gradient flow is 
$
    f(x(t))-f(x^*)=\Order{\frac{1}{t}}
$
(cf.~\cite{W18}).
Therefore, the convergence rate of the optimization method can be intuitively understood from its differential equation.
Su et al.~\cite{SBC16} investigated a second-order differential equation $ \ddot{x} + \frac{3}{t} \dot{x} + \nabla f(x) = 0 $ that corresponds to Nesterov's accelerated gradient method for convex functions (NAG-c)~\cite{N83}. 
The convergence rates of NAG-c and its continuous-time limit are $\Order{{1}/{n^2}}$ and $\Order{{1}/{t^2}}$, respectively. 
This seminal work was followed by several studies~\cite{ACR19, AP16,SZ21,SRBd17,SDJS2022,SDSJ19,USM2023NeurIPS,ZMSJ18,KWB15} (see, e.g., \cite{GMMQS17,RLS18,REQS22,MSZ18,CEOR18} for other studies on the numerical analysis approach to continuous optimization problems).
However, the second-order differential equation asymptotically tends toward a conservative Hamiltonian system, which appears to contradict the better rate than the gradient flow. 

To overcome this difficulty, Scieur et al.~\cite{SRBd17} demonstrated that the NAG method for strongly convex (NAG-sc) functions can be regarded as a linear multistep (LM) method (cf.~\cite{HNW1987}) for the gradient flow~\eqref{eq:gf}.
Furthermore, it was demonstrated that the NAG-sc is zero-stable and consistent, and its acceleration can be attributed to the larger step sizes allowed by the strong stability of the LM method.
Additionally, they applied a similar LM interpretation to NAG-c.
Although this interpretation provides several useful insights, it is not entirely correct from a numerical analysis perspective because LMs are expressed by fixed recurrence formulas, which are consistent with NAG-sc whose coefficients are fixed, except for the step sizes. 
However, because NAG-c includes a time-dependent coefficient, $(n-1)/(n+2)$ (see~\eqref{eq:Nesterov_acc}), the NAG-c cannot be expressed as an LM as long as the step size is fixed.
This study aims to demonstrate that this shortcoming can be addressed by utilizing the concept of {\em variable step size} LMs (VLMs) (see \cref{sec:vlm} for details on VLMs). 
This approach provides a novel perspective that enhances the working environment and extends the scope of NAG-type methods.

Next, we present our target and notations.
Let $f$ be an $L$-smooth convex function.
NAG-c is expressed as follows:
\begin{equation} \label{eq:Nesterov_acc}
    \begin{cases}
        {\displaystyle y_{n+1}=x_n-s\nabla f(x_n),}\\
        {\displaystyle x_{n+1}=y_{n+1}+\frac{n-1}{n+2}(y_{n+1}-y_{n}).}
    \end{cases}
\end{equation}
By removing $y_n$ from~\eqref{eq:Nesterov_acc}, we obtain a new two-step recurrence formula for $x_n$:
  \begin{equation}\label{eq:NAG_multistep}
    x_{n+2} - \frac{2n+3}{n+3} x_{n+1} + \frac{n}{n+3} x_n
    = -s \frac{2n+3}{n+3} \nabla f(x_{n+1}) + s\ \frac{n}{n+3} \nabla f(x_n).
  \end{equation}
This is the main objective of this study. 

The contributions of this study are as follows:
(i) We demonstrate that~\eqref{eq:NAG_multistep} is a stable and consistent VLM (\cref{sec:nag-vlm}). 
Furthermore, we show that the VLM allows {\em linearly increasing} step sizes, which explains why NAG-c achieves a convergence rate of $\Order{1/n^2}$~\cite{N2004}, 
even though (in the current view) it is derived from the gradient flow whose convergence rate is $\Order{1/t}$. 
This provides a simpler and more intuitive perspective of acceleration than second-order differential equation interpretations. 
In this study, we introduce a novel technique for analyzing the absolute stability of VLMs. 
(ii) Considering the contributions in (i), the best VLM among the stable and consistent two-step VLMs is investigated. 
We prove that NAG-c is optimal (\cref{sec:optimality}).
(iii) Considering the contributions in (ii), we must move beyond the scope of stable and consistent VLMs to outperform NAG-c.
Therefore, we remove the consistency condition and optimize the parameters in the two-step VLM for ill-conditioned objective functions (\cref{sec:beyondNAG}).
According to numerical experiments, the proposed method outperforms NAG-c and converges rapidly.

\section{VLMs}\label{sec:vlm}

We introduce several definitions and properties of VLMs based on~\cite{HNW1987,HW1996}.
The approximations $x_{n} \simeq x(t_n)$ of the integral curve $x(t)$ are computed on a grid $t_n$. The variable step size indicates that $h_n := t_{n+1} -t_{n}$ is not constant.

\begin{definition}{\label{def:variableLM}}
    The $k$-step VLM for $\dot{x}(t) = g(x(t))$ is defined as follows:
    \begin{equation}{\label{def:eq:vlm}}
        x_{n+k}+\sum_{j=0}^{k-1}\alpha_{j,n} x_{n+j}=h_{n+k-1}\sum_{j=0}^k \beta_{j,n}g(x_{n+j}),
    \end{equation}
    where $\alpha_{j,n},\ \beta_{j,n}$ are constants that depend on $w_i=h_{i}/h_{i-1}\;(i=n+1,...,n+k-1)$.
\end{definition}
If $\beta_{k,n} = 0$, we can compute $x_{n+k}$ using its explicit form, and the method is said to be explicit. 
Otherwise, we must solve a nonlinear equation to obtain a new point $ x_{n+k}$. 
Although implicit methods are useful in numerical analysis and have been extensively studied, 
we consider only explicit methods in subsequent sections, considering their application to optimization. 

\subsection{Stability}{\label{subsec:stability}}

  When dealing with LMs, we typically consider the following two types of stability: for $\dot{x}(t) = 0$ and Dahlquist's test equation $\dot{x}(t) = \lambda x(t)$, where $ \lambda $ is a complex number.
  
  The stability of $\dot{x}(t) = 0$ is referred to as \emph{zero-stability}, which is defined as follows.
  
  \begin{definition}{\label{def:stable}}
    Let $A_n \in \R^{ k \times k  }$ be a companion matrix defined by 
    \[
        A_n=
            \begin{bmatrix}
                -\alpha_{k-1,n}&\cdots&\cdots&-\alpha_{1,n}&-\alpha_{0,n}\\
                1&0&\cdots &0&0\\
                &1& \ddots &\vdots& \vdots\\
                &&\ddots& 0 &\vdots\\
                &&&1&0                
            \end{bmatrix}.
    \]
    The VLM~\eqref{def:eq:vlm} is zero-stable if 
    there exists a positive real number $M$ such that
    $
        \|A_{n+l}...A_{n+1}A_{n}\|\le M
    $
    holds for all $n,l\ge 0$. 
  \end{definition}

The following lemma is useful in confirming zero-stability. 
Here, by assuming consistency of order $0$ (see~\cref{subsec:consistency} for details), we reduce the size of matrix $A_n$ by one. 

\begin{lemma}[\protect{cf.~\cite[Theorem~5.6]{HNW1987}}]\label{lem:zero-stability of vlm}
    Let the method~\eqref{def:eq:vlm} be consistent of order $p\ge 0$.
    A companion matrix $A_n^* \in \R^{ (k-1) \times (k-1) }$ is defined as follows:
        \[
            A_n^*=
                \begin{bmatrix}
                    -\alpha_{k-2,n}^*&-\alpha_{k-3,n}^*& \cdots &-\alpha_{1,n}^*&-\alpha_{0,n}^*\\
                    1 & 0 & \cdots & 0&0\\
                     & 1 & \ddots & \vdots & \vdots\\
                     &  & \ddots & 0 &\vdots\\
                     &  &  & 1& 0 
                \end{bmatrix},
        \]
        where $\alpha^*_{k-2,n}=1+\alpha_{k-1,n}$, $\alpha_{0,n}^*=-\alpha_{0,n}$, $\alpha^*_{k-j-1,n}-\alpha^*_{k-j,n}=\alpha_{k-j,n}$ $(j=2,...,k-1)$.
        Then, the method~\eqref{def:eq:vlm} is zero-stable if and only if the following conditions hold:
        \begin{itemize}
        \item[(i)] there exists $M_1 > 0 $, for all $n,l \in \Z_{\ge 0} $, $ \norm*{ A^*_{n+l}...A^*_{n+1}A^*_{n} }\le M_1 $, and
        \item[(ii)] there exists $M_2 > 0 $, for all $n,l \in \Z_{\ge 0} $, $ \norm*{ e_{k-1}^\trans \sum_{j=n}^{n+l}\prod_{i=n}^{j-1}A_{i}^*} \le M_2 $,
    \end{itemize}
        where $e_{k-1} = (0,...,0,1)^\trans \in \mathbb{R}^{k-1}$.
\end{lemma}

  The second stability, referred to as \emph{absolute stability}, is defined for Dahlquist's test equation $\dot{x}(t) = \lambda x(t) \; (\Re{\lambda}<0)$. 
  However, the absolute stability of VLMs has not been extensively studied, with only a few case-specific results available (cf.~\cite{CMR1993}). 

  In this study, the following proposition, which is an immediate corollary of \cite[Theorem~3.1]{S2003}, is used to check absolute stability.

  \begin{proposition}\label{prop:ltvd_stability}
    Let $ \{ x_n \}_{n=0}^{\infty} $ is a solution of 
    \begin{equation}\label{eq:ltvd}
        x_{n+k} + \sum_{i=1}^k a_{k-i} (n) x_{n+k-i} = 0, 
    \end{equation}
    where $ a_i \colon \N \to \C \ (i=0,\dots,k-1)$. 
    Suppose that there exists $ a_i \in \C $ such that $ \lim_{n \to \infty } a_i (n) = a_i $ holds $(i=0,\dots,k-1)$, and the roots $ \{ \lambda_i \}_{i=1}^k $ of the polynomial
    $ \lambda^k + \sum_{i=1}^k a_{k-i} \lambda^{k-i}$
    satisfy $ \max_i \abs*{ \lambda_i } < 1 $. 
    Then, $ \lim_{n \to \infty } \abs*{ x_{n} } = 0 $ holds. 
  \end{proposition}

\subsection{Consistency}\label{subsec:consistency}

Consistency is also essential for numerical integration methods and is related to local errors.
When a method exhibits consistency, a sufficiently small step size $h_n$ leads to $x_n \to x(t_n)$.
\begin{definition}{\label{def:consistency}}
  The method~\eqref{def:eq:vlm} is consistent of order $p$ if 
  \begin{equation}\label{eq:cond_consistency}
      q(t_{n+k})+\sum_{j=0}^{k-1} \alpha_{j,n}q(t_{n+j})=h_{n+k-1}\sum_{j=0}^k \beta_{j,n}q'(t_{n+j})
  \end{equation}
  holds for all grids $\{ t_i \}$ and polynomials $q(t)$ of degree $\le p$. 
\end{definition}

An optimization algorithm in the form of VLMs can be regarded as a numerical integration method, 
if it satisfies consistency and zero-stability, 
which are sufficient conditions for convergence (cf.~\cite[Theorem~5.8]{HNW1987}). 

\section{Interpretation of the NAG-c as a VLM}
\label{sec:nag-vlm}

This section demonstrates that NAG-c can be regarded as a stable and consistent VLM. 
Surprisingly, even when the step size $h_n$ increases linearly, the VLM corresponding to NAG-c is absolutely stable. 
This observation provides an intuitive understanding of the accelerated convergence rate $ O (1/n^2) $ of NAG-c. 

We consider the two-step VLM
\begin{equation}\label{def:eq:vlm-nag}
\begin{aligned}
  &x_{n+2}-(5-3w_{n+1})x_{n+1}+(4-3w_{n+1})x_n\\
  &\quad =h_{n+1}\left(\frac{w_{n+1}-1}{w_{n+1}}\right)\left((20-12w_{n+1})g(x_{n+1})-(16-12w_{n+1})g(x_{n})\right).
\end{aligned}
\end{equation}
The application of this method to the gradient flow~\eqref{eq:gf} is identical to the NAG-c method~\eqref{eq:NAG_multistep} when the step size satisfies $h_n = \frac{s}{4}(n+4)$. 

\subsection{Consistency and zero-stability}

The following theorems show the consistency and zero-stability of \eqref{def:eq:vlm-nag}. 

\begin{theorem}{\label{thm:consistency of vlm}}
  The method~\eqref{def:eq:vlm-nag} is consistent of order 1. 
\end{theorem}

\begin{proof}
    In view of \cref{def:consistency}, it is sufficient to confirm \eqref{eq:cond_consistency} for $ q (t) = 1 $ and $ q (t) = t - t_{n+1} $. 
    When $q(t)= 1$,  
    \[
        1-(5-3w_{n+1})+4-3w_{n+1}=0
    \] 
    holds. 
    When $q(t)=t - t_{n+1}$, the left-hand side of \eqref{eq:cond_consistency} becomes
    \begin{align}
        & q ( t_{n+2} ) - (5-3w_{n+1}) q (t_{n+1}) + (4-3w_{n+1}) q (t_n) \\
        &\quad = h_{n+1} - (4-3w_{n+1}) h_n =h_{n+1}\left(1-\frac{4-3w_{n+1}}{w_{n+1}}\right)=4h_{n+1}\frac{w_{n+1}-1}{w_{n+1}}
    \end{align}
    and the right-hand side of \eqref{eq:cond_consistency} can be simplified into
    \[
        h_{n+1}\frac{w_{n+1}-1}{w_{n+1}}\left((20-12w_{n+1})-(16-12w_{n+1})\right)
        =4h_{n+1}\frac{w_{n+1}-1}{w_{n+1}}. 
    \]
    Therefore, this method is consistent of order 1.
\end{proof}

\begin{theorem}\label{thm:zero-stability of vlm}
    The method~\eqref{def:eq:vlm-nag} is zero-stable if and only if the step size ratio $ w_n $ satisfies 
    \begin{itemize}
        \item[(i)] there exists $M_1 > 0 $, for all $n,l \in \Z_{\ge 0} $, $ \abs*{ \prod_{ i = n }^{ n + l } \paren*{ 3 w_{i+1} - 4 } } \le M_1 $, and
        \item[(ii)] there exists $M_2 > 0 $, for all $n,l \in \Z_{\ge 0} $, $ \abs*{ \sum_{ j = n }^{ n + l } \prod_{ i = n }^{ j - 1 } \paren*{ 3 w_{i+1} - 4 } } \le M_2 $. 
    \end{itemize}
    In particular, the method~\eqref{def:eq:vlm-nag} is zero-stable if the step sizes are chosen as follows:
    \begin{itemize}
        \item exponentially growing step sizes: $ h_n = h_0 \gamma^n $, where $ \gamma \in (1,5/3) $, or 
        \item linearly growing step sizes: $ h_n = h_0 + n \gamma $, where $ \gamma \in (0, \infty) $.
    \end{itemize}
\end{theorem}

\begin{proof}
    Since the method~\eqref{def:eq:vlm-nag} is a two-step method, 
    $
    A_n^* = \alpha_{0,n}^* = -\alpha_{0,n} = - (4-3w_{n+1}).
    $
    holds. 
    Therefore, \cref{lem:zero-stability of vlm} shows the first half of the theorem. 

    When the step sizes are chosen as $ h_n = h_0 \gamma^n $ with a fixed constant $ \gamma \in (1,5/3) $, 
    $ w_{n+1} = \gamma $ holds. 
    Since this implies that $ \abs*{ 4 - 3 w_{n+1} } = \abs*{ 4 - 3 \gamma } < 1 $ holds for all $n \in \N$, the method~\eqref{def:eq:vlm-nag} is zero-stable: 
    \begin{align}
        \abs*{ \prod_{ i = n }^{ n + l } \paren*{ 3 w_{i+1} - 4 } } &= \abs*{4 - 3 \gamma}^{l+1} < 1, \\
        \abs*{ \sum_{ j = n }^{ n + l } \prod_{ i = n }^{ j - 1 } \paren*{ 3 w_{i+1} - 4 } } &\le \sum_{ j = n }^{ n + l } \abs*{4 - 3 \gamma}^{j-n} < \frac{1}{1- \abs*{4-3\gamma}}.
    \end{align}

    When the step sizes are chosen as $ h_n = h_0 + n \gamma $ with a fixed constant $ \gamma \in (0,\infty) $, $ w_{n+1} = 1 + \frac{1 }{n + h_0/\gamma} $ and $ 3 w_{n+1} - 4 = - 1 + \frac{3}{n + h_0 / \gamma} $ hold. 
    Then, since $ \abs*{ 3 w_{n+1} - 4 } \le 1 $ holds for sufficiently large $n$, the condition (i) holds. 
    To show the condition (ii), we introduce $ \hat{\gamma} = h_0 / \gamma $. 
    Then, we have $ 3 w_{n+1} - 4 = - \frac{ n - 3 +\hat{\gamma} }{n + \hat{\gamma}} $ and 
    \begin{align}
        \sum_{ j = n }^{ n + l } \prod_{ i = n }^{ j - 1 } \paren*{ 3 w_{i+1} - 4 }
        &= \sum_{ j = n }^{ n + l } \paren*{ (-1)^{ j-1-n } \prod_{ i = n }^{ j - 1 } \frac{ i - 3 + \hat{\gamma}  }{i + \hat{\gamma}} } \\
        &= \sum_{ j = n }^{ n + l } (-1)^{ j-1-n } \frac{ \paren*{n - 3 + \hat{\gamma}} \paren*{n - 2 + \hat{\gamma}} \paren*{n - 1 + \hat{\gamma}} }{ \paren*{j - 3 + \hat{\gamma}} \paren*{j - 2 + \hat{\gamma}} \paren*{j - 1 + \hat{\gamma}} }, 
    \end{align}
    which is bounded due to \cref{lem:tech1}
\end{proof}

When the step size satisfies $h_n = \frac{s}{4}(n+4)$ (i.e., when $w_{n} = \frac{n+4}{n+3}$ holds), 
the zero-stability holds based on this theorem. 
Therefore, it is reasonable to regard NAG-c as a numerical integration method for the gradient flow~\eqref{eq:gf}. 

The accelerated phenomenon can be intuitively understood. 
Recall that the solution of the gradient flow~\eqref{eq:gf} satisfies $ f(x(t)) - f (x^{\ast}) = \Order{\frac{1}{t}} $. 
Moreover, the step size $ h_n = \frac{s}{4}(n+4) $ that corresponds to NAG-c implies that
\begin{align}
  t_n = \sum_{j=0}^{n-1} t_{j+1}-t_j = \sum_{j=0}^{n-1} h_j = \sum_{j=0}^{n-1} \frac{s}{4}(j+4) = \Order{n^2}.
\end{align}
In summary, we obtain an accelerated convergence rate of $\Order{\frac{1}{n^2}}$. 

However, in view of \cref{thm:zero-stability of vlm}, 
the step size can grow exponentially, which violates the well-known lower bound $ \Order{\frac{1}{n^2}} $ of first-order optimization methods for convex functions~\cite{D17}. 
This contradiction is caused by considering only zero-stability. 
As shown in the next section, the step size can increase only linearly if absolute stability is considered.

\subsection{Absolute stability}{\label{subsec:stable region}}
  
  In this section, we derive the stability region for the VLM~\eqref{def:eq:vlm-nag} by using \cref{prop:ltvd_stability}. 
  To this end, we apply the method~\eqref{def:eq:vlm-nag} to Dahlquist's test equation $ \dot{x} = \lambda x $, and obtain
\begin{equation}
    \begin{aligned}
    &x_{n+2}-(5-3w_{n+1})x_{n+1}+(4-3w_{n+1})x_n\\
    &\quad = \lambda h_{n+1}\left(\frac{w_{n+1}-1}{w_{n+1}}\right)\left((20-12w_{n+1}) x_{n+1}-(16-12w_{n+1}) x_{n}\right).
    \end{aligned}
\end{equation}
Since $h_{n+1}(w_{n+1}-1)/w_{n+1} = h_{n+1}-h_n$ holds, 
the assumption of \cref{prop:ltvd_stability}, the convergence of coefficients, is ensured for at most linearly growing step sizes (note that $h_{n+1} - h_n $ diverges when the step size grows exponentially). 
Therefore, we introduce $\mu_n \coloneqq \lambda (h_{n+1} - h_{n})$ to handle linearly increasing step sizes and obtain 
\begin{equation}\label{eq:vlm-nag_test}
    x_{n+2} - \paren*{ 5 - 3 w_{n+1} } \paren*{ 1 + 4 \mu_n } x_{n+1} + \paren*{ 4 - 3 w_{n+1} } \paren*{ 1 + 4 \mu_n } x_{n} = 0.
\end{equation}

For later use, we define the absolute stability for the general case. 
To this end, we apply the method~\eqref{def:eq:vlm} (with $k=2$ as an example) to Dahlquist's test equation, and obtain
  \begin{equation}\label{eq:vlm_test}
        x_{n+2} + \alpha_{1,n} x_{n+1} + \alpha_{0,n} 
        = \mu_n \frac{w_{n+1}}{w_{n+1}-1}( \beta_{2,n}x_{n+2} + \beta_{1,n}x_{n+1} + \beta_{0,n}x_n ). 
  \end{equation}
  Consequently, we use \cref{prop:ltvd_stability} to define the absolute stability as follows. 

  \begin{definition}\label{def:vlm_absolute_stability}
    VLM~\eqref{def:eq:vlm} is absolutely stable,  
    if $ \mu_n $, $ \alpha_{j,n} $ and $ \frac{w_{n+1}}{w_{n+1}-1} \beta_{j,n} $ converges to $ \mu $, $ \alpha_j $ and $ \widehat{\beta}_j $ as $n \to \infty $, respectively, and
    the roots of the polynomial
    \begin{align}\label{def:characteristic polynomial}
        R (z)
        &=\left(1-\mu \widehat{\beta}_k \right) z^k + \sum_{j=0}^{k-1} \paren*{ \alpha_j - \mu \widehat{\beta}_j } z^{j} 
    \end{align}
    lie in the unit circle. 
  \end{definition}

  \begin{remark}
    \Cref{def:vlm_absolute_stability} is a natural extension of the usual definition of absolute stability for the fixed step size LMs to the case of linearly increasing step sizes. 
    As well as the fixed step size cases, the boundary of the stability region can be written by using the root locus curve (see, e.g., \cite[V.1]{HW1996}). 
    However, this definition has two minor flaws. 
    First, this definition does not determine stability over the boundary.
    For practical purposes, however, whether the boundary is stable or unstable is not important (except for the origin, which corresponds to the zero-stability). 
    Second, the condition considered here is a sufficient condition for the numerical solution of VLM~\eqref{def:eq:vlm} applied to Dahlquist's test equation to not diverge, but it is not necessarily a necessary condition. 

    \Cref{prop:ltvd_stability} can also be used to analyze the absolute stability of typical VLMs such as the Adams method (see \cref{app:adams} for details on the two-step explicit Adams method). 
  \end{remark}

    \begin{figure}[htp]
    \centering
        \begin{tikzpicture}
            \begin{axis}[width = 10cm,height=5cm,compat=newest,xlabel={$\mathrm{Re}$},ylabel={$\mathrm{Im}$},axis equal,axis lines = middle, axis line style = {->},yticklabels={},xmin = -0.35,xmax=0.05,ymin=-0.11,ymax=0.11,xtick = {-0.3333,-0.25,0},xticklabels={$-\frac{1}{3}$,$-\frac{1}{4}$,},ytick={-0.1,0,0.1},yticklabels={$-0.1$,,$0.1$}]
            \addplot[mark = none,black] table {data/vlmnagsr.dat};
            \end{axis}
        \end{tikzpicture}
        \caption{Stability region boundary for the VLM~\eqref{def:eq:vlm-nag} provided by the condition in \cref{thm:vlm stable region}.}
        \label{fig:stable_region_z}
    \end{figure}
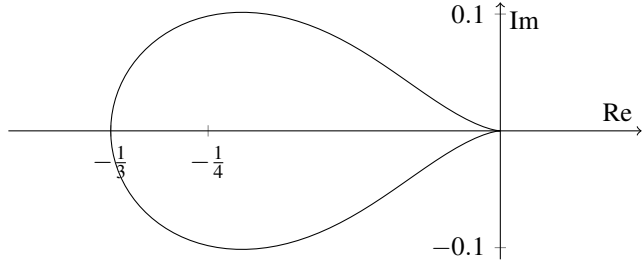

  Based on this definition of stability, we can plot the stability region for the VLM (see \cref{fig:stable_region_z} for the method~\eqref{def:eq:vlm-nag}). 
  After the asymptotic behavior of step size $h_n$ is determined, the absolute stability of the method~\eqref{def:eq:vlm-nag} can be confirmed by verifying that $\mu$ is contained in the stability region. 
  From the stability region, we demonstrate that the method~\eqref{def:eq:vlm-nag} with a linearly growing step size is absolutely stable.
  
  \begin{theorem}{\label{thm:vlm stable region}}
        The method~\eqref{def:eq:vlm-nag} is absolutely stable if and only if $\{ h_n \}_{n=0}^{\infty}$ satisfies the following conditions{\em :}
        \begin{itemize}
            \item The sequence $ \{ \mu_n \}_{n=0}^{\infty} $ converges as $n\to \infty $ (the limit is denoted by $\mu$){\em ;}
            \item The roots of polynomial $ z^2 - 2 (1 + 4\mu) z + (1+4\mu) $ lie in the unit circle. 
        \end{itemize}
        In particular, when $\mu \in \R $, the second condition is equivalent to $ \mu \in \paren*{ - \frac{1}{3} , 0 } $. 
  \end{theorem}

\begin{proof}
Note that, when $ \lim_{n \to \infty } \mu_n = \mu $ holds, $ \lim_{n \to \infty} w_n = 1 $ is satisfied. 
Then, we have
\begin{align}
    \lim_{n \to \infty} \alpha_{1,n} &= -2,&
    \lim_{n \to \infty} \alpha_{0,n} &= 1,\\
    \lim_{n \to \infty} \frac{w_{n+1}}{w_{n+1}-1} \beta_{1,n} &= 8,&
    \lim_{n \to \infty} \frac{w_{n+1}}{w_{n+1}-1} \beta_{0,n} &= -4. 
\end{align}
Therefore, in this case, the polynomial~\eqref{def:characteristic polynomial} is $ z^2 - 2 ( 1 + 4\mu) + (1+4\mu) $, 
which proves the first half of the theorem. 
To prove the second half, we use \cref{def:characteristic polynomial} and obtain the necessary and sufficient condition ``$ \abs*{1 + 4\mu} < 1$ and  $\abs*{1 + 4\mu} < 1 + 2\mu $''.
This is equivalent to $ -\frac{1}{3} < \mu < 0 $. 
\end{proof}

\section{Optimality of NAG-c within a natural class of VLMs}
\label{sec:optimality}

    We deal with the general form of two-step VLMs,
    \begin{equation}{\label{def:general-vlm}}
        x_{n+2}+a_{n}x_{n+1}+b_{n}x_n=h_{n+1}\left(c_ng(x_{n+1})-d_ng(x_n)\right), 
    \end{equation}
    and present a sufficient condition to prove its convergence rate using the so-called Lyapunov argument. 
    Moreover, we prove the optimality of NAG-c among a natural class of methods. 

\subsection{Consistency and stability}{\label{sec:ext:consistency}}

    We assume that the VLM~\eqref{def:general-vlm} is consistent of order $1$, zero-stable, and absolutely stable. 
    The consistency is expressed as
    \begin{align}{\label{eq:order of 0}}
        a_n &= - b_n - 1 ,\\
        {\label{eq:order of 1}}
        d_{n} &= c_n+\frac{b_n}{w_{n+1}}-1.  
    \end{align}
    Under this assumption, two free parameters $b_n$ and $ c_n$ are required to determine the method~\eqref{def:general-vlm}. 
    In addition, similar to the proof of \cref{thm:zero-stability of vlm}, 
    the method~\eqref{def:general-vlm} is zero-stable if and only if the following two conditions are satisfied:
    \begin{itemize}
        \item[(i)] there exists $M_1 > 0 $, for all $n,l \in \Z_{\ge 0} $, $ \abs*{ \prod_{ i = n }^{ n + l } b_i } \le M_1 $, and
        \item[(ii)] there exists $M_2 > 0 $, for all $n,l \in \Z_{\ge 0} $, $ \abs*{ \sum_{ j = n }^{ n + l } \prod_{ i = n }^{ j - 1 } ( - b_i ) } \le M_2 $. 
    \end{itemize}
    Furthermore, in view of \cref{def:vlm_absolute_stability}, we assume 
    \begin{align}
        \lim_{n \to \infty} b_n &= b,&
        \lim_{n \to \infty} \frac{w_{n+1}}{w_{n+1}-1} c_n &= \widehat{c},&
        \lim_{n \to \infty} \frac{w_{n+1}}{w_{n+1}-1} d_n &= \widehat{d}.
    \end{align}
    Because of the above first condition and \eqref{eq:order of 0}, $ \lim_{n\to \infty} a_n = -b - 1$ holds. 
    Moreover, due to the condition on $d_n$ and \eqref{eq:order of 1}, we have $ b = 1$. 
    Then, the characteristic polynomial~\eqref{def:characteristic polynomial} is
    $ z^2 - \paren*{2+ \widehat{c} \mu } z + 1 + \widehat{d} \mu$, 
    and the roots of this polynomial lie in the unit circle if and only if both of $ \abs*{1 + \widehat{d} \mu} < 1$ and $\abs*{2 + \widehat{c} \mu} < 2 + \widehat{d} \mu$ hold. 
    In summary, regarding the absolute stability, we assume
    \begin{align}
        \lim_{n \to \infty} \frac{b_n-w_{n+1}}{w_{n+1}-1} &= \widehat{d} - \widehat{c},&
        \lim_{n \to \infty} \frac{w_{n+1}}{w_{n+1}-1} c_n &= \widehat{c}, \label{eq:general-vlm-as1}\\
        0 < \widehat{d} &< \widehat{c},&
        - \frac{4}{ \widehat{c} + \widehat{d} } < \mu &< 0. \label{eq:general-vlm-as2}
    \end{align}
    
\subsection{Convergence rate by Lyapunov functions}{\label{subsec:convergence rate}}
    We derive the conditions under which the method~\eqref{def:general-vlm} for the gradient flow~\eqref{eq:gf} has a Lyapunov function that reveals a convergence rate. 
    To simplify this discussion, we introduce a sequence $\{ \xi_n \} $ and a real constant $r$ using the relation
    $ b_n = \frac{\xi_n}{\xi_{n+1} + r}$. 
    This coincides with NAG-c when $\xi_n = n$ and $r = 2$. 
    We define the Lyapunov function for \eqref{def:general-vlm} as follows:
    \begin{equation}{\label{def:Lyapunov function}}
        \begin{split}
            \lyp(n) &=B_n(f(x_n)-f(x^*))\\
            &\quad +\frac{1}{2}\|\xi_n(x_{n+1}-x_n) + h_{n+1}(\xi_{n+1}+r)d_n\nabla f(x_n)+r(x_{n+1}-x^*)\|^2,
        \end{split}
    \end{equation}
    where $ B_n=A_{n+1}\xi_n $ and $ A_{n+1}=h_{n+1}(\xi_{n+1}+r)c_{n}-h_{n+2}(\xi_{n+2}+r)d_{n+1} $. 
    The following lemma provides the conditions for this function to be monotonically non-increasing. 
    
    \begin{lemma}{\label{lem:lyapunov decrease monotonically}}
        Suppose that $a_n=-1-b_n, \ b_n=\frac{\xi_n}{\xi_{n+1}+r}$ and $\xi_n > 0$. 
        If 
        \begin{align}
            {\label{eq:lyapunov decrease monotonically cond1}}
                B_n &\ge 0, \\
            {\label{eq:lyapunov decrease monotonically cond2}}
                B_{n+1}-B_n &\ge 0,\\
            {\label{eq:lyapunov decrease monotonically cond3}}
                B_{n+1}-B_n-rA_{n+1} &\le 0,\\
            {\label{eq:lyapunov decrease monotonically cond4}}
            \frac{A_{n+1}}{2} - \frac{r}{2L} - h_{n+1}(\xi_{n+1}+r)d_n + \frac{A_{n+1}h_{n+1}^2(\xi_{n+1}+r)^2d_n^2L}{2B_n} &\le 0
        \end{align}
        hold, then $\{x_n\}$ that is generated by the method~\eqref{def:general-vlm} for the gradient flow~\eqref{eq:gf} satisfies
        $
            \lyp(n+1)\le \lyp(n)
        $.
    \end{lemma}

\begin{proof}
From the $L$-smoothness and convexity of the function, the following inequalities hold for arbitrary $x,y \in \mathbb{R}^d$ (cf.~\cite[Theorem~2.1.5]{N2004}):
\begin{align}
    \label{eq:L-smooth and convex1}
    &\frac{1}{2L}\|\fgrad(x)-\fgrad(y)\|^2\le\frac{1}{L}\|\fgrad(x)-\fgrad(y)\|^2\le \pro{\fgrad(x)-\fgrad(y)}{x-y},\\
    \label{eq:L-smooth and convex2}
    &f(x)\le f(y)+\langle \nabla f(x),x-y\rangle -\frac{1}{2L}\|\nabla f(x)-\nabla f(y)\|^2.
\end{align}

    First, we introduce $z_n$ as follows:
    \[
        z_n=\xi_n(x_{n+1}-x_n)+h_{n+1}(\xi_{n+1}+r)d_n\nabla f(x_n)+r(x_{n+1}-x^*).
    \]
    From the scheme~\eqref{def:general-vlm} and $b_n = \frac{\xi_n}{\xi_{n+1} + r}$, we obtain
    \begin{equation}\label{eq:generalLM2_update}
        (\xi_{n+1}+r)(x_{n+2}-x_{n+1})-\xi_n(x_{n+1}-x_n)=-h_{n+1}(\xi_{n+1}+r)(c_n\nabla f(x_{n+1})-d_n\nabla f(x_{n})).
    \end{equation}
    Then, we observe that
    \begin{align}
        z_{n+1}-z_n 
        &= \xi_{n+1}(x_{n+2}-x_{n+1}) + h_{n+2}(\xi_{n+2}+r)d_{n+1}\nabla f(x_{n+1})+r(x_{n+2}-x^*)\\
        & \qquad -\xi_n(x_{n+1}-x_n) - h_{n+1}(\xi_{n+1}+r)d_n \nabla f(x_n)-r(x_{n+1}-x^*)\\
        & = (\xi_{n+1} + r)(x_{n+2}-x_{n+1}) -\xi_n(x_{n+1}-x_n) - h_{n+1}(\xi_{n+1}+r)d_n\nabla f(x_n)\\
        &\qquad + h_{n+2}(\xi_{n+2}+r)d_{n+1}\nabla f(x_{n+1})\\
        & = -\left(h_{n+1}(\xi_{n+1}+r)c_n-h_{n+2}(\xi_{n+2}+r)d_{n+1} \right)\nabla f(x_{n+1})\\
        & = -A_{n+1}\fgrad(x_{n+1}). \label{eq:diff of zn}
    \end{align}
    The third equality follows from~\eqref{eq:generalLM2_update}.
    We evaluate the difference $ \lyp(n+1)-\lyp(n) $ with~\eqref{eq:diff of zn};
    \begin{align}
        \lyp(n+1)-\lyp(n)
        &=B_n(f(x_{n+1})-f(x_n))+(B_{n+1}-B_n)(f(x_{n+1})-f(x^*))\\
        &\hspace{150pt}+\frac{1}{2}\langle z_{n+1}-z_{n},z_{n+1} -z_{n} + 2z_{n}  \rangle\\
        &=B_n(f(x_{n+1})-f(x_n))+(B_{n+1}-B_n)(f(x_{n+1})-f(x^*))\\
        &\hspace{100pt} - A_{n+1}\langle \fgrad(x_{n+1}),z_{n}\rangle + \frac{A_{n+1}^2}{2}\|\fgrad(x_{n+1}) \|^2.
    \end{align}
    For the most right-hand side of the equation above, we evaluate the first, second, and third terms in order.
    We apply~\eqref{eq:L-smooth and convex2} to the first term.
    Then, from the condition~\eqref{eq:lyapunov decrease monotonically cond1}, we obtain
    \begin{equation}{\label{eq:first term}}    
        B_n (f(x_{n+1})-f(x_n)) \le B_n \langle \nabla f(x_{n+1}), x_{n+1}-x_{n}\rangle -\frac{B_n}{2L}\| \nabla f(x_{n+1}) - \nabla f(x_{n})\|^2.
    \end{equation}
    For the second term, based on~(\ref{eq:lyapunov decrease monotonically cond2}), we obtain 
    \begin{align}{\label{eq:second term}}    
        \notag
        &(B_{n+1}- B_n) (f(x_{n+1})-f(x^*)) \\
        & \qquad \le (B_{n+1}- B_n) \langle \nabla f(x_{n+1}), x_{n+1}-x^*\rangle -\frac{(B_{n+1}- B_n)}{2L}\| \nabla f(x_{n+1})\|^2.
    \end{align}
    For the third term, based on the definition of $z_n$, we obtain
    \begin{align}{\label{eq:third term}}
        \notag
        A_{n+1} \langle \nabla f(x_{n+1}), z_n\rangle
        &= A_{n+1}\xi_n \langle \nabla f(x_{n+1}), x_{n+1} - x_n\rangle\\
        \notag
        &\quad + A_{n+1}h_{n+1}(\xi_{n+1} + r)d_n\langle \nabla f(x_{n+1}),\nabla f(x_{n})\rangle\\
        &\quad + rA_{n+1} \langle \nabla f(x_{n+1}), x_{n+1} - x^*\rangle.
    \end{align}
    By combining all three terms, we obtain 
    \begin{align}
        &\lyp(n+1) - \lyp(n)\\ 
        &\ \le (B_n - A_{n+1}\xi_n)\langle \nabla f(x_{n+1}), x_{n+1} - x_n\rangle \\
        &\quad + (B_{n+1}- B_n -r A_{n+1}) \langle \nabla f(x_{n+1}), x_{n+1}-x^*\rangle -\left( \frac{B_{n+1}}{2L} - \frac{A_{n+1}^2}{2}\right)\|\fgrad (x_{n+1})\|^2 \\
        &\quad + \paren*{ \frac{B_n}{L} - A_{n+1}h_{n+1}(\xi_{n+1}+r)d_n} \langle \fgrad(x_{n+1}),\fgrad(x_n)\rangle
        - \frac{B_n}{2L}\|\fgrad (x_n)\|^2.
    \end{align}
    From the the condition~\eqref{eq:lyapunov decrease monotonically cond3}, inequality~\eqref{eq:L-smooth and convex1}, and definition of $B_n$, 
    \begin{align}
        &\lyp(n+1) - \lyp(n)\\
        &\le -\left( \frac{B_n}{2L}- \frac{A_{n+1}^2}{2} + \frac{r A_{n+1}}{2L}\right)\|\fgrad (x_{n+1})\|^2\\
        &+ \left(  \frac{B_n}{L} - A_{n+1}h_{n+1}(\xi_{n+1}+r)d_n\right)\langle \fgrad(x_{n+1}),\fgrad(x_n)\rangle
        - \frac{B_n}{2L}\|\fgrad (x_n)\|^2\\
        &\le -\frac{B_n}{2L}\left\|\fgrad(x_n)-\left(1-\frac{A_{n+1}h_{n+1}(\xi_{n+1}+r)d_nL}{B_n}\right)\fgrad(x_{n+1})\right\|^2 \\
        &+\left(
        -\frac{B_n}{2L} + \frac{A_{n+1}^2}{2} - \frac{r A_{n+1}}{2L}
            +\frac{B_n}{2L}\left(1-\frac{A_{n+1}h_{n+1}(\xi_{n+1}+r)d_nL}{B_n}\right)^2
        \right)\|\fgrad(x_{n+1})\|^2,
    \end{align}
    holds.
    We now rewrite the coefficient of $\|\fgrad(x_{n+1})\|^2$ to obtain
    \begin{align}
        &-\frac{B_n}{2L} + \frac{A_{n+1}^2}{2} - \frac{r A_{n+1}}{2L}
            +\frac{B_n}{2L}\left(1-\frac{A_{n+1}h_{n+1}(\xi_{n+1}+r)d_nL}{B_n}\right)^2\\
        &\quad = A_{n+1}\left(\frac{A_{n+1}}{2} - \frac{r}{2L} - h_{n+1}(\xi_{n+1}+r)d_n + \frac{A_{n+1}h_{n+1}^2(\xi_{n+1}+r)^2d_n^2L}{2B_n} \right).
    \end{align}
    From the conditions~\eqref{eq:lyapunov decrease monotonically cond1}, \eqref{eq:lyapunov decrease monotonically cond4} and $\xi_n > 0$, the coefficient of $\|\fgrad(x_{n+1})\|^2$
    is negative; therefore,
    $
        \lyp(n+1)-\lyp(n)\le 0
    $
    holds.
\end{proof}

    According to \cref{lem:lyapunov decrease monotonically}, if the conditions \eqref{eq:lyapunov decrease monotonically cond1}, \eqref{eq:lyapunov decrease monotonically cond2}, \eqref{eq:lyapunov decrease monotonically cond3} and \eqref{eq:lyapunov decrease monotonically cond4} are satisfied for all $n$, 
    $
        \lyp(n)\le \lyp(0)
    $
    holds.
    Recall the definition \eqref{def:Lyapunov function}. 
    To evaluate $\lyp(0)$, we select $x_1 = x_0$ as the starting point and $d_0$ = 0 and $\xi_0 = 0$ as the initial parameter.
    These settings are reasonable under consistency conditions \eqref{eq:order of 0} and \eqref{eq:order of 1} because the update formula becomes
    $
        x_2 =x_1 - h_1 \nabla f (x_1).
    $
    This equation is identical to the explicit Euler method with a step size of $h_1$.
    Subsequently, we obtain 
    \begin{equation}
        B_n(f(x_n) - f(x^*)) \le \lyp(n) \le \lyp(0) \le \frac{r^2}{2}\|x_1-x^*\|^2, 
    \end{equation}
    which reveals the convergence rate $ f (x_n) - f(x^*) = \Order{r^2/B_n}$. 

    Conditions~\eqref{eq:lyapunov decrease monotonically cond1} and \eqref{eq:lyapunov decrease monotonically cond2} are natural because they are necessary for the above argument to reveal the convergence rate using the Lyapunov function~$ \lyp $. 
    The remaining conditions~\eqref{eq:lyapunov decrease monotonically cond3} and \eqref{eq:lyapunov decrease monotonically cond4} also appear naturally to prove that the function $\lyp$ is monotonically non-increasing. 

The convergence rate of NAG-c can be shown by \cref{lem:lyapunov decrease monotonically}. In \eqref{def:eq:vlm-nag}, which has a step size $h_n = an + 3a$, the parameters can be written as 
\begin{align}
    \xi_n &= n,&
    r &= 2,&
    c_n &= \frac{8n+12}{(n+3)(n+4)},&
    d_n &= \frac{4n}{(n+3)(n+4)}.
\end{align}
Then, $B_n$ and $A_{n+1} $ are written as $ B_n = 4an(n+2)$ and $ A_{n+1} = 4a(n+2) $, respectively. 
Therefore, the condition of Lemma~\ref{lem:lyapunov decrease monotonically} can be written as follows:
\begin{align}
    4an(n+2)&\ge 0,&
    4a(2n+3)&\ge 0,&
    -4a &\le 0,&
    \frac{\left(4 L a - 1\right) \left(2 L a n + 1\right)}{L} &\le 0.
\end{align}
The above conditions are satisfied for all $n\ge 0$ if $0 \le a \le \frac{1}{4L}$.
Therefore, 
$
    4an(n+2)(f(x_n)-f(x^*) \le \lyp(n) \le \lyp(0)\le 2\|x_1 - x^*\|^2
$
holds.

\subsection{NAG-c optimality}{\label{subsec:optimality}}
    We prove the optimality of the NAG-c among methods in the form of~\eqref{def:general-vlm} when consistency of order $1$, zero-stability, absolute stability, and the convergence rate $O(1/n^2)$ that is ensured by \cref{lem:lyapunov decrease monotonically} are assumed. 

    \Cref{lem:lyapunov decrease monotonically} yields the convergence rate $\Order{r^2/B_n}$. 
    Hence, we consider the case in which $B_n$ becomes $\Theta(n^2)$. 
    In the following, we briefly explain that $B_n$ is $\Order{{\xi_n}^2}$. 
    Therefore, it is sufficient to consider the case $\xi_n = \Order{n} $ when $h_n = \Order{n}$. 
    Based on the definition $ B_n = A_{n+1} \xi_n $, it is sufficient to confirm that the order of $A_n$ is $\Order{\xi_n}$. 
    This is satisfied because $h_{n+1}c_n$ and $h_{n+2}d_{n+1}$ are bounded because of the condition~\eqref{eq:general-vlm-as1}:
    as $n \to \infty$, $ h_{n+1} c_n = \paren*{ h_{n+1} - h_n } \frac{w_{n+1}}{w_{n+1}-1} c_n \to \frac{\mu}{\lambda} \widehat{c} $ and 
    $ h_{n+2} d_n = w_{n+2} \paren*{ h_{n+1} - h_n } \frac{w_{n+1}}{w_{n+1}-1} d_n \to \frac{\mu}{\lambda} \widehat{d} $ hold. 
    
    Let $\xi_n = n$ and $h_n = an + (1+r)a$, where $ a $ is a positive real constant,
    because the $\mathrm{o}(n)$ terms in $\xi_n$ and $h_n$ do not contribute to the dominant term of the convergence rate from the previous argument. 
    Subsequently, we determine the $\mathrm{o}(n)$ terms for convenient calculation. 
    Moreover, considering the relation $ b_n = \frac{\xi_n}{\xi_{n+1} + r}$
    we let the coefficient of $ n $ in $ \xi_n$ be $1$, which has no asymptotic effect.   
    
    \begin{theorem}\label{thm:Nesterov-optim}
        Let the step size $h_n = an + (1+r)a$ and $\xi_n = n$.
        Suppose that the method~\eqref{def:general-vlm} for the gradient flow~\eqref{eq:gf} is consistent of order $1$, zero-stable, and absolutely stable.
        Furthermore, suppose that the method satisfies the convergence rate $\Order{1/n^2}$ ensured by \cref{lem:lyapunov decrease monotonically}. 
        Then, the parameters that minimize the coefficient of $1/n^2$ in $r^2/B_n$ are those for NAG-c.
    \end{theorem}

    \begin{proof}
        We now verify inequalities~\cref{eq:lyapunov decrease monotonically cond1,eq:lyapunov decrease monotonically cond2,eq:lyapunov decrease monotonically cond3,eq:lyapunov decrease monotonically cond4} in \cref{lem:lyapunov decrease monotonically} and derive the conditions on parameters.
        For the sake of simplicity, we replace $c_n$ with $\frac{\bar{c}_n}{(n+1+r)(n+2+r)}$. 
        Then, $A_{n+1}$ and $B_n$ can be written as 
        \begin{align}
            B_n &= an((r+2)n + (r+2)^2 + \bar{c}_n - \bar{c}_{n+1}),\\
            A_{n+1} &= a((r+2)n + (r+2)^2 + \bar{c}_n - \bar{c}_{n+1}).
        \end{align}
        Using this expression, we can rewrite \eqref{eq:lyapunov decrease monotonically cond1}--\eqref{eq:lyapunov decrease monotonically cond3} as follows:
        \begin{align}
            \label{eq:cond1 in optimality proof}
            (r+2)n + (r+2)^2 + \bar{c}_n - \bar{c}_{n+1} &\ge 0,\\
            \label{eq:cond2 in optimality proof} 
            -n\bar{c}_n + (2n+1)\bar{c}_{n+1} - (n+1)\bar{c}_{n+2} + 2(r+2)n + r^2+5r+6 &\ge 0,\\
            \label{eq:cond3 in optimality proof}
            n(\bar{c}_n -2\bar{c}_{n+1} + \bar{c}_{n+2} + r^2-4) - r\bar{c}_n - (r+1)\bar{c}_{n+1} +\bar{c}_{n+2} + r^3 + 3r^2 -r -6 &\ge 0.
        \end{align}
        With respect to~\eqref{eq:lyapunov decrease monotonically cond4}, we obtain
        \begin{align}
            \label{eq:cond4 in optimality proof}
            \frac{\gamma_2 n^2 + \gamma_1 n + \gamma_0}{2Ln} \le 0,
        \end{align}
        where $\gamma_2,\gamma_1,\gamma_0$ is defined as follows:
        \begin{align}
            \gamma_2 &= La(r+2)(La(r+2)+3),\\
            \gamma_1 &= -La(2La(r+2) + 1)\bar{c}_n -La\bar{c}_{n+1} + 2L^2a^2(r+1)(r+2)^2 
            + La(3r+4)(r+2)-r,\\
            \gamma_0 &= L^2a^2 \bar{c}_n^2 - 2L^2a^2(r+2)(r+1)\bar{c}_n + L^2a^2(r^4 + 6r^3 +13r^2 + 12 r +4).
        \end{align}
    
        We now determine which terms have the largest order of $n$.
        Let $\bar{c}_n=\Theta(n^p)$. If $p>1$, then the largest order is $n^{2p}$ from $\gamma_0$, and the coefficient is $L^2a^2$. This term is positive and we cannot prove the convergence as $n \to \infty$.
        If $p<1$, then the largest order is $n^{2}$ from $\gamma_2 n^2$ and the coefficient of the term is $\gamma_2$. Since $\gamma_2$ is positive, we cannot prove the convergence as $n \to \infty$.
        Therefore, we can prove the convergence if and only if $p=1$. Therefore, we let $\bar{c}_n=\phi n+\psi$ and determine $\psi$ using $d_0 = 0$.
        From $\bar{c}_n =\phi n+\psi$, $d_n$ can be written as
        \[
            d_n = - \frac{(r+2-\phi)n + r^2 + 3r + 2 - \psi}{(n+1+r)(n+2+r)}.
        \]
        Then, we set $\psi = r^2 + 3r + 2$.
        The conditions \cref{eq:cond1 in optimality proof,eq:cond2 in optimality proof,eq:cond3 in optimality proof,eq:cond4 in optimality proof} are further simplified into
        \begin{align}   
            \label{eq:cond1 in optimality proof2}
            (r+2)n + (r+2)^2 - \phi&\ge 0,\\
            \label{eq:cond2 in optimality proof2}
            (2r+4) n + r^2 + 5r +6 -\phi &\ge 0,\\
            \label{eq:cond3 in optimality proof2}
            (r^2-4)n + r^3 + 3r^2 -r\phi -r + \phi -6 &\ge 0,\\
            \label{eq:cond4 in optimality proof2}
            \frac{a}{2}((r+2)^2-\phi ) -\frac{r}{2L} + \frac{an}{2}(La(r-\phi+2)^2 + (3r - 2\phi +6))&\le 0.
        \end{align}
        Due to the conditions (\ref{eq:cond3 in optimality proof2}) and (\ref{eq:cond4 in optimality proof2}), parameters $r,\phi,a$ are required to satisfy
        \begin{align}
            \label{eq:constraint1 in optimality proof}
            r^2-4\ge 0,\\
            \label{eq:constraint2 in optimality proof}
            La(r-\phi+2)^2 + (3r - 2\phi +6)\le 0.
        \end{align}
        We minimize the convergence rate $r^2/B_n$ under these constraints. 
        Precisely speaking, we minimize the coefficient of $1/n^2$ in $r^2/B_n$, which is $r^2/a(r+2)$. The optimization problem can be written as follows:
        \[
            \min_{a,r,\phi} \frac{r^2}{a(r+2)} \quad \text{s.t. } (\ref{eq:constraint1 in optimality proof}), (\ref{eq:constraint2 in optimality proof}).
        \]
        By solving this problem, we obtain $r = 2,\phi = 8$ and $a = \frac{1}{4L}$; then, $\psi = 12$. These parameters are identical to those used in the method \eqref{def:eq:vlm-nag}.
    \end{proof}

\section{Towards improved methods beyond the NAG-c}
\label{sec:beyondNAG}

    In the previous section, we considered the method~\eqref{def:general-vlm} assuming a consistency of order 1, zero-stability, and absolute stability.  
    In this setting, NAG-c is optimal in terms of the convergence rate. 
    Therefore, we now consider a broader class of methods in the form of~\eqref{def:general-vlm}, in which we assume the consistency of order $0$ instead of $1$. 
    
    We evaluate the convergence rate of the VLM for the linear ordinary differential equation (ODE),
    \[
        \dot{x}(t) = A x(t).
    \]
    In the optimization settings, this ODE appears when the objective function is quadratic implying that $A$ is identical to $- \nabla^2 f$.
    Because $A$ is a symmetric matrix, its eigenvalue decomposition is expressed as
    \begin{equation}{\label{eq:dahlquist in optimization}}
        \dot{x}_i(t) = \lambda_i x_i(t) \quad(\text{$\lambda_i $ is one of the eivenvalues of $-\nabla^2 f$}),
    \end{equation}
    which is identical to Dahlquist's test equation. 
    Because we assume that $f$ is convex and $L$-smooth, $\lambda_i \in [-L,0]$ holds. 
    For brevity, we consider the case $h_n = \frac{n}{L}$ so that $\mu_n\in [-1,0]$. 
    Because we do not assume the consistency of order $1$, the step size can be freely scaled. 

\subsection{Convergence analysis for Dahlquist's test equation}
We introduce $ \widehat{c}_n = \frac{w_{n+1}-1}{w_{n+1}}c_n $ and $\widehat{d}_n = \frac{w_{n+1}-1}{w_{n+1}}d_n$ for simplicity. 
Subsequently, the characteristic polynomial for \eqref{eq:vlm_test} can be written as
\begin{equation}{\label{eq:2-step polynomial}}
    R_n(z) = z^2  - \left( 1 + b_n + \mu_n \widehat{c}_n \right)z + \left( b_n + \mu_n \widehat{d}_n \right).
\end{equation}
Let $r^{(n)}$ denote the maximum absolute value of the roots of $R_n(z)$. 
Intuitively, 
a smaller $ r^{(n)} $ is better. 

NAG-c corresponds to the case, $b_n = \widehat{d}_n= 4-3w_{n+1}, \ \widehat{c}_{n}= 5-3w_{n+1}$. 
From this, we obtain 
\begin{equation}
    r^{(n)} = \begin{cases}
        {\displaystyle \sqrt{ \left(4 - 3 w_{n+1} \right) \left( 1 + \mu_n \right) }} & (\mu_n \in [-1,\mu_n^*]), \\[5pt]
        {\displaystyle \frac{5 - 3 w_{n+1}}{2} \left( 1 + \mu_n + \sqrt{ \left( 1+\mu_n \right) \left( \mu_n - \mu_n^* \right) } \right)} & ( \mu_n \in (\mu_n^*,0 ]),
    \end{cases}
\end{equation}
where $ \mu_n^* = -\frac{(3w_{n+1}-3)^2}{(5-3w_{n+1})^2} $ (see \cref{fig:roots-NAG-c}). 
When $ \mu_n \in (-1,\mu_n^*) $, the roots are conjugate complex numbers. 
Otherwise, the roots are real numbers. 

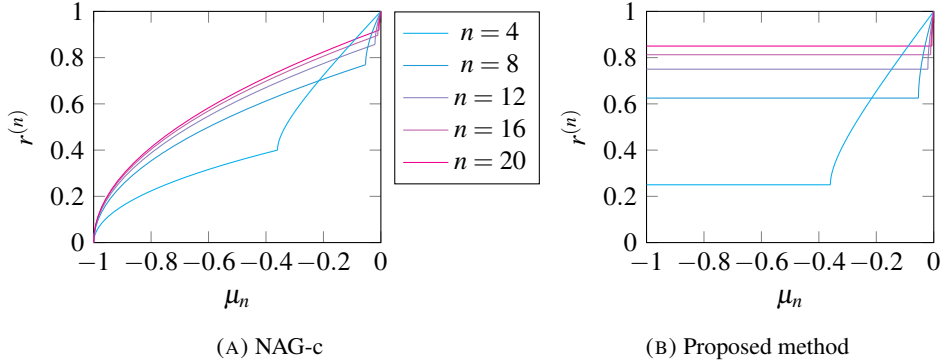
\begin{figure}[htbp]
\centering
    \begin{minipage}{0.58\textwidth}    
    \centering
    \begin{tikzpicture}
        \begin{axis}[width=5.4cm,compat=newest,xlabel=$\mu_n$,ylabel={$r^{(n)}$},xmin=-1,xmax=0,ymin=0,ymax=1,legend style={
            at={(1.3,1)},
            anchor=north}]
            \foreach [count=\i from 0,evaluate=\i as \redfrac using \i*100/4] \n in {4,8,12,16,20}
            {
                \edef\temp{\noexpand\addplot[color=magenta!\redfrac!cyan,domain=0:sqrt(1-9/(2*\n-3)^2)] ({x^2-1},{x*sqrt((\n-3)/\n)});
                \noexpand\addlegendentry{$n=\n$}}
                \temp
            }
            \foreach [count=\i from 0,evaluate=\i as \redfrac using \i*100/4] \n in {4,8,12,16,20}
            {
                \edef\temp{\noexpand\addplot[color=magenta!\redfrac!cyan,domain=0:3/(2*\n-3)] ({-9/(2*\n-3)^2+x^2},{(2*\n-3)/\n/2*(1-9/(2*\n-3)^2+x^2+x*sqrt(1-9/(2*\n-3)^2+x^2))});}
                \temp
            }
        \end{axis}
    \end{tikzpicture}
    \subcaption{NAG-c}
    \label{fig:roots-NAG-c}
    \end{minipage}
    \begin{minipage}{0.41\textwidth}
    \centering
    \begin{tikzpicture}
        \begin{axis}[width=5.4cm,compat=newest,xlabel=$\mu_n$,ylabel={$r^{(n)}$},xmin=-1,xmax=0,ymin=0,ymax=1]
            \foreach [count=\i from 0,evaluate=\i as \redfrac using \i*100/4] \n in {4,8,12,16,20}
            {
            \edef\temp{\noexpand\addplot[color=magenta!\redfrac!cyan,domain=-1:-9/(2*\n-3)^2] {(\n-3)/\n};}
                \temp
            }
            \foreach [count=\i from 0,evaluate=\i as \redfrac using \i*100/4] \n in {4,8,12,16,20}
            {
                \edef\temp{\noexpand\addplot[color=magenta!\redfrac!cyan,domain=0:3/(2*\n-3)] ({-9/(2*\n-3)^2+x^2},{(1+((\n-3)/\n)^2 + ((2*\n-3)/\n)^2 * ( -9/(2*\n-3)^2+x^2 + x*sqrt(1-9/(2*\n-3)^2+x^2) ) )/2});}
                \temp
            }
        \end{axis}
    \end{tikzpicture}
    \subcaption{Proposed method}
    \label{fig:roots-proposed}
    \end{minipage}
    \caption{Maximum absolute values of the roots of the polynomial $R_n $ with respect to $\mu_n$. }
    \label{fig:roots}
\end{figure}

\Cref{fig:roots-NAG-c} indicates that the NAG-c converges quickly for large eigenvalue components and slowly for small eigenvalue components. 

\subsection{Proposed method}

Based on the findings of the previous section, it is intuitive to construct a new VLM based on the optimization problem 
\begin{align}
    \min_{b_n,\widehat{c}_n,\widehat{d}_n} \max_{\mu_n \in [-1,0]} r^{(n)}. 
\end{align}
However, this problem is meaningless because $r^{(n)} =1$ holds for $\mu_n = 0$. 
Therefore, we restrict the range of $\mu_n$. 
Consequently, we consider only complex roots based on our NAG-c observations for the following two reasons. 
First, because $ \lim_{n\to \infty} \mu_n^* = 0 $ holds, 
the eigenvalue component corresponds to the complex roots for a sufficiently large $n$ when considering one fixed eigenvalue. 
Second, the absolute values of complex roots can be written in a simple form to analytically solve the optimization problem. 

Furthermore, to clarify the relationship with the NAG-c case, 
we assume that the region of $\mu_n$, in which the complex roots appear, is the same as that of NAG-c. 
That is, we assume that
\begin{align}
\label{eq:constraint1}
    (1-\widehat{c}_n + b_n)^2 - 4(b_n - \widehat{d}_n) &= 0,\\
\label{eq:constraint2}    
    (1 + \widehat{c}_n \mu_n^* + b_n)^2 - 4(b_n + \widehat{d}_n \mu_n^*) &=0. 
\end{align}

As a result, we consider the optimization problem
\begin{equation}
    \min_{b_n,\widehat{c}_n,\widehat{d}_n} \max_{\mu_n \in [-1,\mu_n^*]} r^{(n)} \qquad \text{s.t. }  \eqref{eq:constraint1}, \eqref{eq:constraint2}.
\end{equation}
Because the absolute value of the roots can be written as $ \sqrt{ b_n + \mu_n d_n } $ for $\mu \in 
[-1,\mu_n^*]$ and the square root is monotonic, 
the above problem is equivalent to the optimization problem
\begin{align}{\label{optimization problem for proposed method}}
    \min_{b_n,\widehat{c}_n,\widehat{d}_n} \max \left\{b_n - \widehat{d}_n,b_n + \mu_n^* \widehat{d}_n \right\} \qquad \mathrm{s.t.}\; \eqref{eq:constraint1}, \eqref{eq:constraint2}.
\end{align}

By solving this problem, we obtain the proposed method (see \cref{app:proposed} for the derivation)
\begin{equation}\label{def:proposed method}
    \begin{split}
    &x_{n+2} - (1 + (4-3w_{n+1})^2)x_{n+1} + (4- 3w_{n+1})^2 x_n\\
    &\quad = h_{n+1} \frac{w_{n+1}-1}{w_{n+1}} (5-3w_{n+1})^2 g(x_{n+1})
    \end{split}
\end{equation}
(see \cref{fig:roots-proposed} for the $ r^{(n)}$ of this method). 

The proposed method satisfies a consistency of order $0$ instead of order $1$. 
In addition, the proposed method is zero-stable (it can be verified in a manner similar to \cref{thm:zero-stability of vlm}).
In this sense, the method is beyond the scope of \cref{thm:Nesterov-optim} and may outperform NAG-c. 
However, we cannot prove that the convergence rate of the proposed method using \cref{lem:lyapunov decrease monotonically}. 
We also cannot prove the absolute stability in the sense of \cref{def:vlm_absolute_stability}. 

\subsection{Numerical experiment}\label{subsec:numerical experiment}

We conducted numerical experiments to compare our proposed method \eqref{def:proposed method} with NAG-c. 
All experiments were performed in Python3.8 and on AMD EPYC 7413 24-Core processors with an RTX A5000 GPU. We fixed the step size $h_n = an + 3a$ and searched for the best parameter $a = i\times 10^{j}\;(i \in \{1,2,3,4,5,6,7,8,9\}, j \in \mathbb{Z})$ for each problem, except for the function $\sum_{i = 1}^6 ix_i^2$. 
With this function, we set $a = \frac{1}{6}$.
We used five test functions:
(a) quadratic functions with a Hilbert matrix, $H \in \mathbb{R}^{10,000\times 10,000}$, (b) one-dimensional (1D) Cahn--Hilliard (CH) problem ($d=999$)
\begin{align}
       \min_{U \in \mathbb{R}^N} & \ \sum_{k=1}^N \left(\frac{U_k^4}{4} - \frac{U_k^2}{2} \right)\Delta x + \sum_{k=1}^{N-1} \frac{1}{2} (\delta^+ U_k)^2 \Delta x + \frac{1}{2} ((\delta^+ U_1)^2 + (\delta^+ U_{N-1})^2)\Delta x\\
       \mathrm{s.t.} & \ U_1 = -1, \quad U_N =1
\end{align}
($N = 1001$, $\Delta x = \frac{1}{N-1}$,  $ \delta^+ U_k = \frac{U_{k+1}-U_k}{\Delta x} $, $x_0 = (-1,-1 + 1/(N-1),...,1-1/(N-1),1)$), 
which is a discretization of the problem
 \begin{align}    
    \min_u & \ \int_0^1 \frac{(u(x))^4}{4} - \frac{(u(x))^2}{2} + \frac{1}{2}\left(\dif{u}{x}\right)^2 \mathrm{d}x, \\
    \mathrm{s.t.} & \ u(0) = -1, \quad u(1)=1,
\end{align}
(c) the LogSumExp function $\sigma\log{\left(\sum_{i=1}^m \exp{((a_i^\trans x - b_i)/\sigma)}\right)}$ ($\sigma = 10, m = 10^5, d = 10^4$, dataset $a_i\;(i=1,...,m)$ is randomly generated from $\mathcal{N}(0,I_n)$, and $b_i = a_i^\trans \zeta + \varepsilon_i$, where the entries of $\zeta\in \mathbb{R}^n$ and $\varepsilon\in \mathbb{R}^m$ are sampled from $\mathcal{N}(0,10)$ and $\mathcal{N}(0,1)$, respectively), 
(d) a logistic regression using realistic datasets
\[
    \min_{x} \frac{1}{m}\sum_{i=1}^m f_i(x) + \lambda \|x\|^2,
\]
where $f_i = \ln{(1+e^{-b_i (a_i^\trans x)})}$, $(a_i,b_i)$ is $i$-th data and $b_i \in \{-1,1\}$
(when we use the {\em real-smi} datasets, $m = 72,309$ and the dimension of $x$ is 20,958, $\lambda = 10^{-10}$), 
and (e) a well-conditioned function.

\begin{figure}[htbp]
    \centering
    \begin{minipage}{0.32\linewidth}
    \centering
        \includegraphics[scale = 0.26]{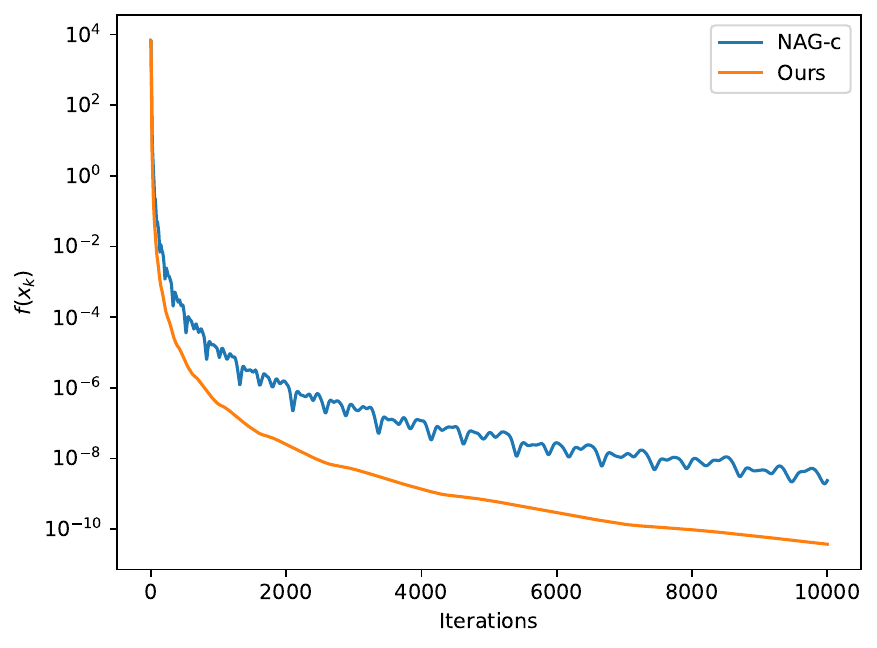}
        \subcaption{$\frac{1}{2}x^\trans H x$}
        \label{fig:subcap:hilbert}
    \end{minipage}
    \begin{minipage}{0.32\linewidth}
        \centering
        \includegraphics[scale = 0.26]{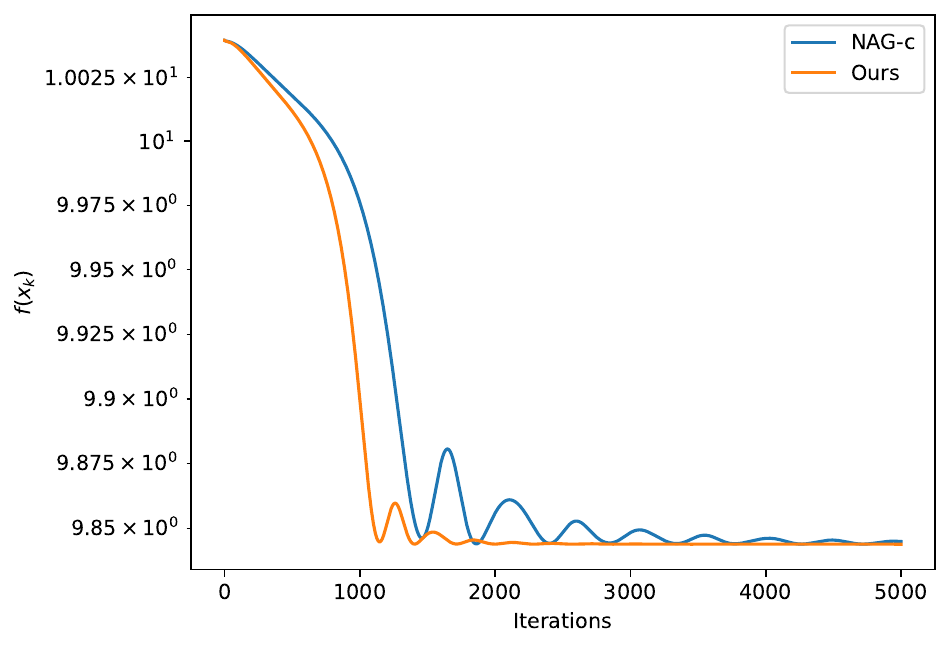}
        \subcaption{1D CH problem}
        \label{fig:subcap:cahn-hilliard}
    \end{minipage}
    \begin{minipage}{0.32\linewidth}
        \centering
        \includegraphics[scale = 0.26]{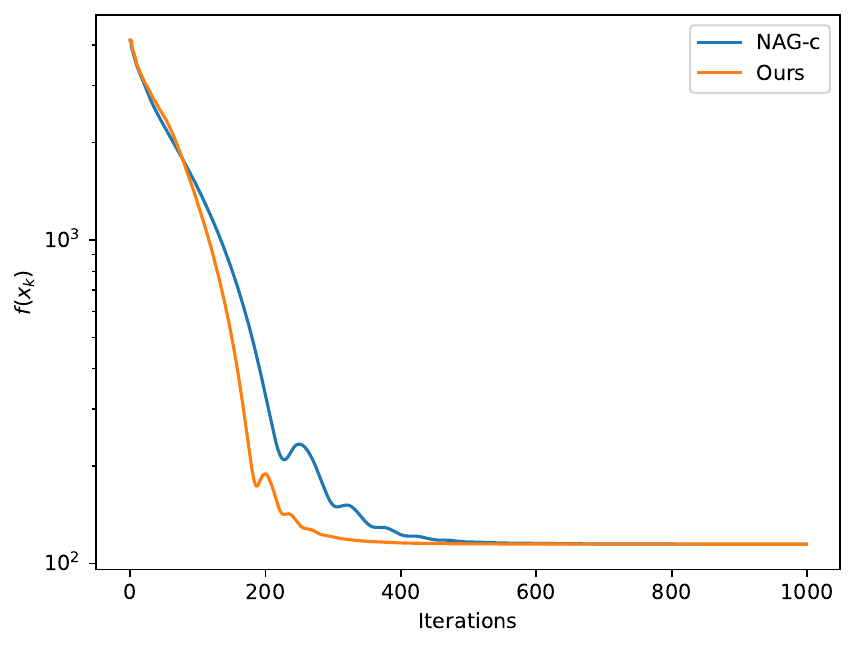}
        \subcaption{LogSumExp function}
        \label{fig:subcap:logsumexp}
    \end{minipage}
    \newline
    \begin{minipage}{0.32\linewidth}
    \centering
        \includegraphics[scale = 0.25]{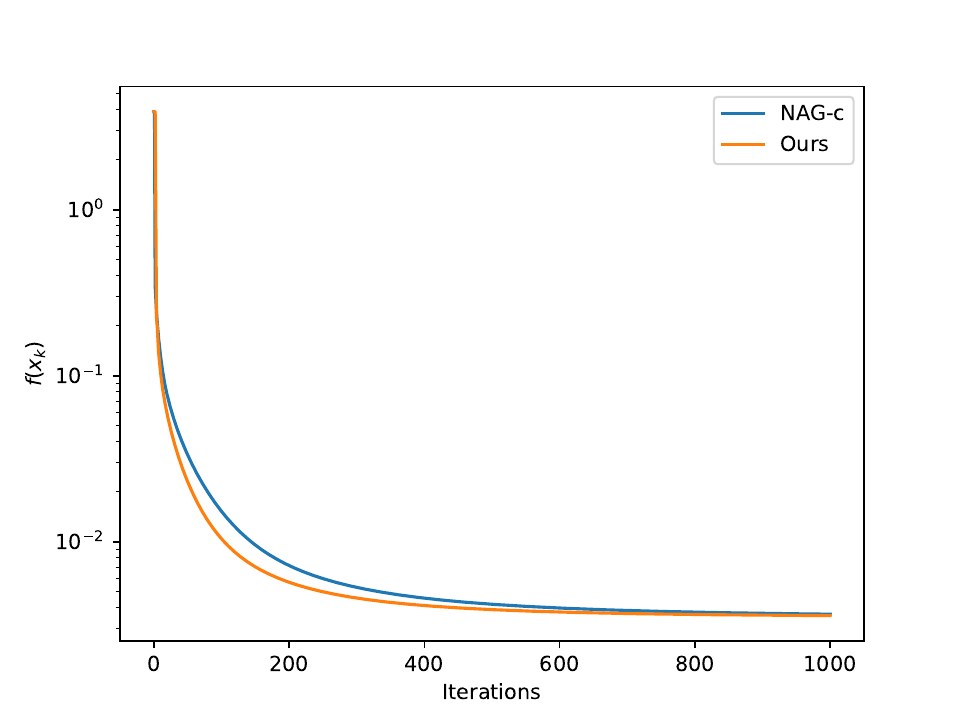}
        \subcaption{Logistic regression}
        \label{fig:subcap:logistic}
    \end{minipage}
    \begin{minipage}{0.32\linewidth}
        \centering
        \includegraphics[scale = 0.25]{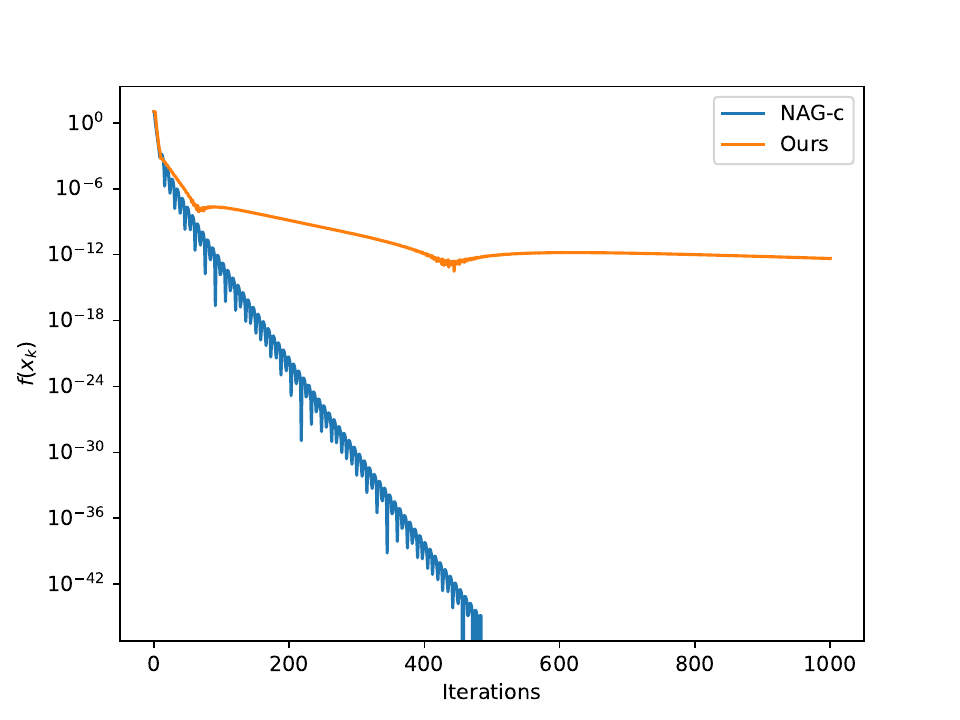}
        \subcaption{$\sum_{i=1}^6 ix_i^2$}
        \label{fig:subcap:simple}
    \end{minipage}
    \caption{Results of numerical experiments for five test functions.}
    \label{fig:numerical experiment}
\end{figure}

The iteration-versus-function values for each of the six test functions are shown in \cref{fig:numerical experiment}.
In \cref{fig:subcap:hilbert}, the proposed method outperforms NAG-c because the quadratic function with the Hilbert matrix is ill-conditioned, owing to the eigenvalues of ill-conditioned functions often being distributed across a large range. Moreover, the proposed method is designed to quickly converge to the optimum for all eigenvalue components. However, when the eigenvalues of the objective functions are concentrated, the proposed method is less effective than NAG-c. This phenomenon can be observed in \cref{fig:subcap:simple}. However, the proposed method is effective when using the restart technique~\cite{OC2015}. We demonstrate this result later. 

According to \cref{fig:subcap:cahn-hilliard}, the Cahn--Hilliard problem, which is also ill-conditioned, the proposed method converges faster than the NAG-c, although the objective function is not quadratic. 
Furthermore, when the objective function is $L$-smooth and not strongly convex (\cref{fig:subcap:logsumexp}), the proposed method is just as effective as NAG-c.
We used real-sim datasets for the logistic regression in \cref{fig:subcap:logistic}, and the proposed method was demonstrated to be as effective as NAG-c under $l_2$-regularization.

Because NAG-c frequently uses a restart technique, we compared its performance. We restarted the method when the function value did not decrease (i.e., $f(x_{n+1}) > f(x_n)$).
In \cref{fig:numerical experiments with restart}e, the proposed method also outperforms NAG-c even when the objective function is not ill-conditioned.  

\begin{figure}[htbp]
\centering
    \begin{minipage}{0.32\linewidth}
        \centering
        \includegraphics[scale = 0.25]{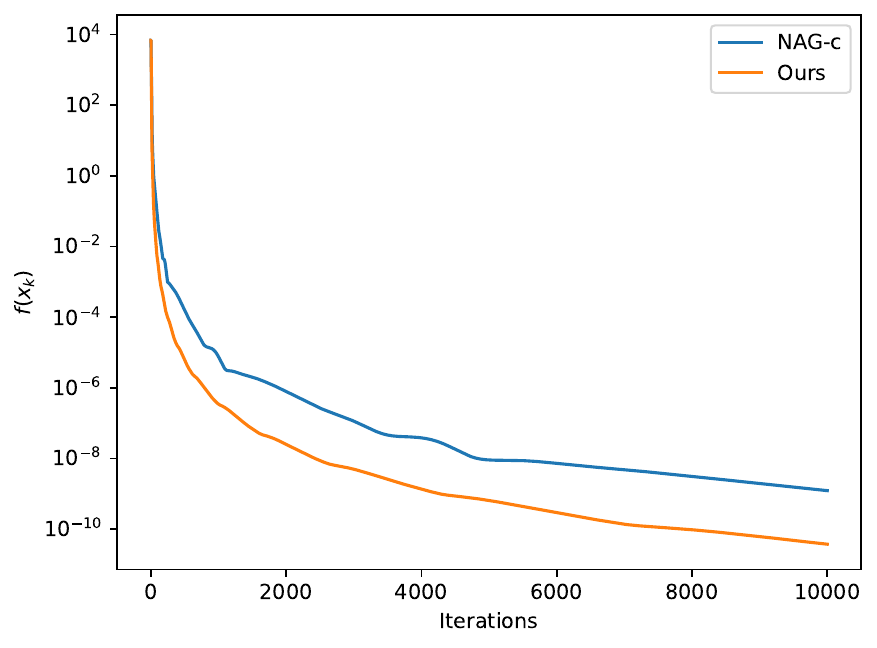}
        \subcaption{$\frac{1}{2}x^\trans H x$}
    \end{minipage}
    \begin{minipage}{0.32\linewidth}
        \centering
        \includegraphics[scale = 0.25]{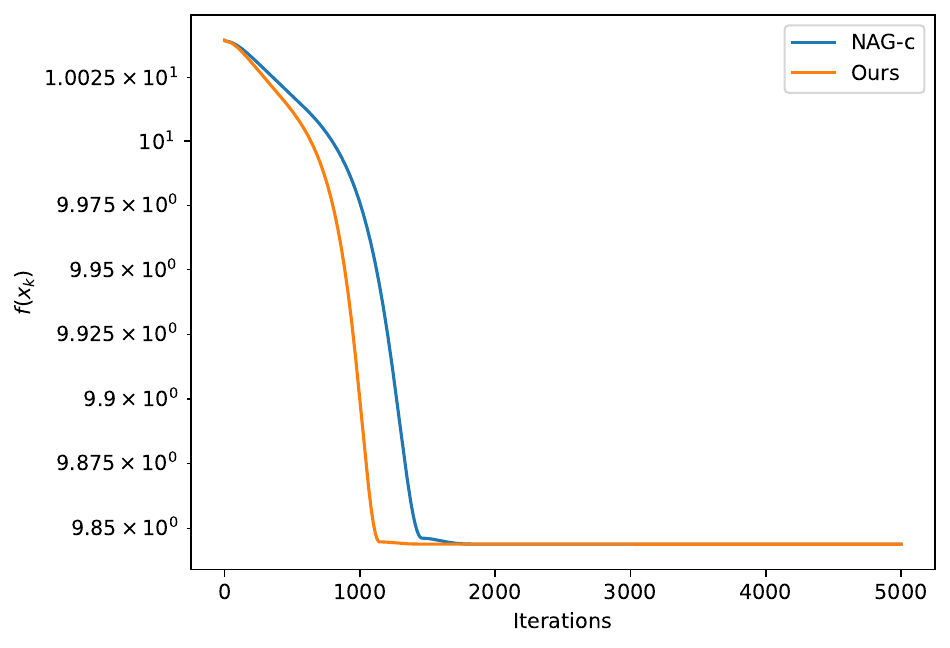}
        \subcaption{1D CH problem}
    \end{minipage}
    \begin{minipage}{0.32\linewidth}
        \centering
        \includegraphics[scale = 0.25]{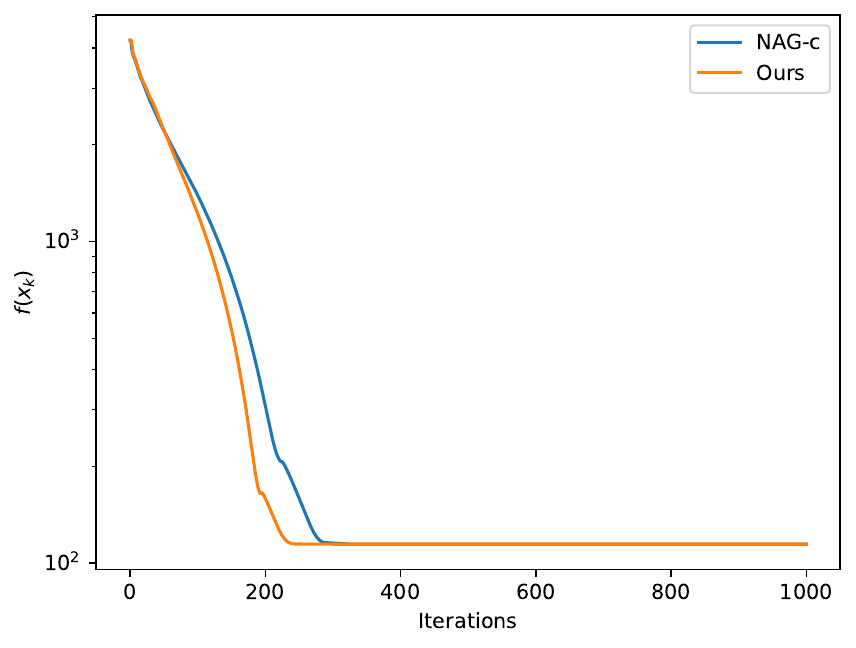}
        \subcaption{LogSumExp function}
    \end{minipage}
    \newline
    \begin{minipage}{0.32\linewidth}
    \centering
        \includegraphics[scale = 0.25]{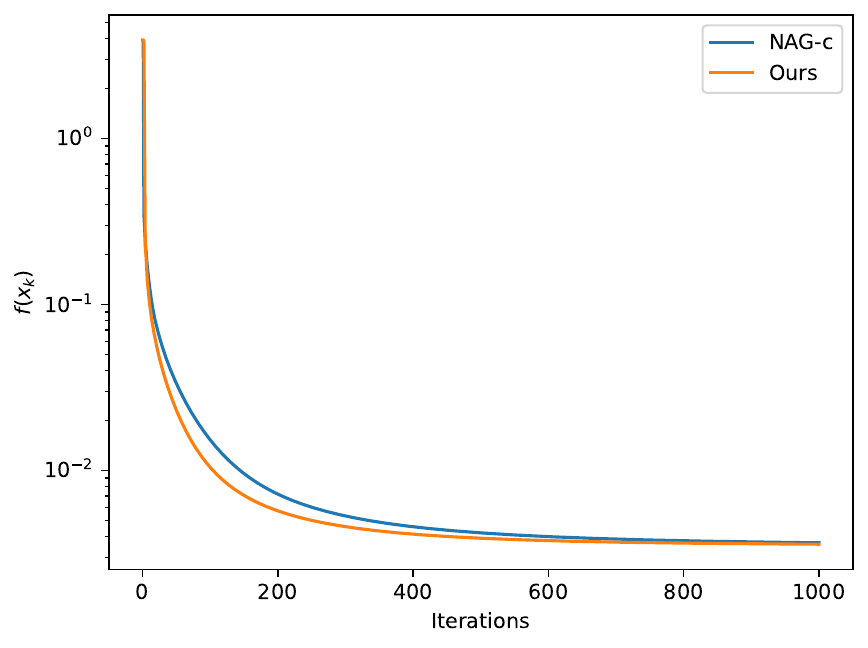}
        \subcaption{Logistic regression}
    \end{minipage}
    \begin{minipage}{0.32\linewidth}
        \centering
        \includegraphics[scale = 0.25]{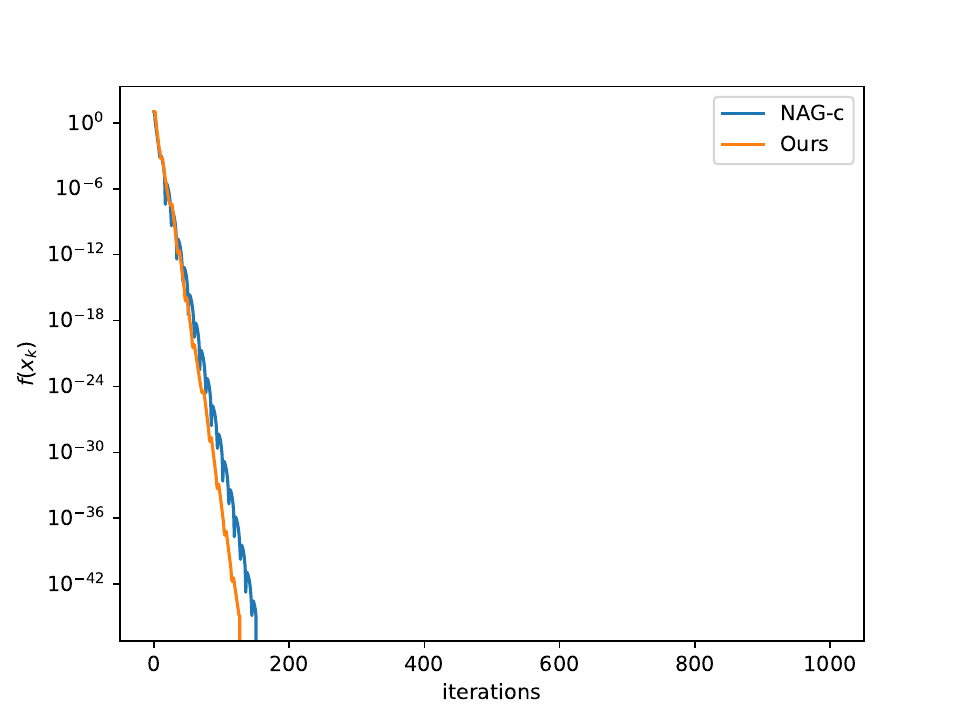}
        \subcaption{$\sum_{i=1}^6 ix_i^2$}
    \end{minipage}
    \caption{Results of numerical experiments with restart.}
    \label{fig:numerical experiments with restart}
\end{figure}

\section{Conclusions and future work}
We analyzed the NAG-c method for $L$-smooth and convex functions from a VLM perspective and demonstrated that it can be viewed as a numerical integration method for the gradient flow.
By considering a general two-step VLM, we derived the conditions under which the algorithms converge at a rate of $O(1/n^2)$.
Furthermore, we proved that one of the optimal algorithms that is consistent of order $1$ and stable is identical to NAG-c.
Finally, we considered the class that is consistent of order $0$ and analyzed the linear ODE case. 
Based on these analyses, we proposed a novel method that outperforms NAG-c when the objective functions are ill-conditioned. 

Future studies must investigate the precise convergence rate of the proposed method. Several assumptions were made with focus on analytically determining the optimal parameters when constructing the proposed method; however, it is unknown whether a better method can be derived without these assumptions.
In this study, only two-step VLMs were considered; however, extensions to three or more steps are also possible. 

\bibliographystyle{abbrv}
\bibliography{references}

\begin{appendix}

\section{Technical Lemmas}

In this section, we introduce several technical lemmas that are used in the proof of theorems in the main part. 

\begin{lemma}\label{lem:tech1}
    Let $ \hat{\gamma} $ be a positive real number. 
    There exists $ M$ such that
    \[ S_{n,l} \coloneqq \abs*{ \sum_{ j = n }^{ n + l } (-1)^{ j-1-n } \frac{ \paren*{n - 3 + \hat{\gamma}} \paren*{n - 2 + \hat{\gamma}} \paren*{n - 1 + \hat{\gamma}} }{ \paren*{j - 3 + \hat{\gamma}} \paren*{j - 2 + \hat{\gamma}} \paren*{j - 1 + \hat{\gamma}} } } \le M \]
    holds for all $ n , l \in \Z_{\ge 0} $. 
\end{lemma}

\begin{proof}
    First, we show the case where $n$ is less than or equal to $3$.
    Since $ \{ 0, 1,2,3 \} $ is finite, we have $ M_0 = \max_{ n, l \in \{ 0, 1,2,3 \} } S_{n,l} < \infty$. 
    When $ l > 3 $ and $ n \in \{ 0,1,2,3\} $ hold, we see
    \begin{align}
        S_{n,l} 
        &\le M_0 + \abs*{ \paren*{n - 3 + \hat{\gamma}} \paren*{n - 2 + \hat{\gamma}} \paren*{n - 1 + \hat{\gamma}} } \abs*{ \sum_{ j = n + 4 }^{ n + l } \frac{ 1 }{ \paren*{j - 3 + \hat{\gamma}}^3 } } \\
        &\le M_0 + \max_{ m \in \setE{0,1,2,3} } \setE*{ \abs*{ \paren*{m - 3 + \hat{\gamma}} \paren*{m - 2 + \hat{\gamma}} \paren*{m - 1 + \hat{\gamma}} } } \sum_{j=0}^{\infty} \frac{1}{j^3}. 
    \end{align}
    The right-hand side is bounded and the bound is denoted by $M_1$.
    
    Second, we consider the case $ n > 3$. 
    Since 
    \begin{align}
        &\frac{ 1 }{ \paren*{j - 3 + \hat{\gamma}} \paren*{j - 2 + \hat{\gamma}} \paren*{j - 1 + \hat{\gamma}} } - \frac{ 1 }{ \paren*{j - 2 + \hat{\gamma}} \paren*{j - 1 + \hat{\gamma}} \paren*{j + \hat{\gamma}} } \\
        &\quad = \frac{ 3 }{ \paren*{j - 3 + \hat{\gamma}} \paren*{j - 2 + \hat{\gamma}} \paren*{j - 1 + \hat{\gamma}} \paren*{j + \hat{\gamma}} }
    \end{align}
    holds, we have
    \begin{align}
        & \sum_{ j = n }^{ n + l }  \frac{ \paren*{-1}^{ j - n  } }{ \paren*{j - 3 + \hat{\gamma}} \paren*{j - 2 + \hat{\gamma}} \paren*{j - 1 + \hat{\gamma}} } \\
        &\quad \le
        \sum_{j = 0}^{ \left\lfloor \frac{l+2}{2} \right\rfloor} \frac{ 3 }{ \paren*{n + 2j - 3 + \hat{\gamma}} \paren*{n + 2j - 2 + \hat{\gamma}} \paren*{n + 2j - 1 + \hat{\gamma}} \paren*{n + 2j + \hat{\gamma}} } \\
        &\quad \quad + \frac{1}{ \paren*{ n + l - 2 + \hat{\gamma } } \paren*{ n + l - 1 + \hat{\gamma } } \paren*{ n + l + \hat{\gamma } } }.
    \end{align}
    In addition, we see
    \begin{align}
        &\sum_{j = 0}^{\infty} \frac{ 3 }{ \paren*{n + 2j - 3 + \hat{\gamma}} \paren*{n + 2j - 2 + \hat{\gamma}} \paren*{n + 2j - 1 + \hat{\gamma}} \paren*{n + 2j + \hat{\gamma}} } \\
        &\quad \le \frac{3}{\paren*{n-3+\hat{\gamma}}^4} + \int_{0}^{\infty} \frac{ 3 }{ \paren*{ 2x + n - 3 + \hat{\gamma} }^4 } \rd x \\
        &\quad = \frac{1}{ \paren*{ n - 3 + \hat{\gamma} }^4 } + \frac{1}{ 2 \paren*{ n - 3 + \hat{\gamma} }^3 }.
    \end{align}
    Therefore, we have 
    \begin{align}
        S_{n,l}
        &\le \frac{ \paren*{n - 2 + \hat{\gamma}} \paren*{n - 1 + \hat{\gamma}} }{ \paren*{ n - 3 + \hat{\gamma} }^3 } + \frac{ 3 \paren*{n - 2 + \hat{\gamma}} \paren*{n - 1 + \hat{\gamma}} }{ 2\paren*{ n - 3 + \hat{\gamma} }^2 }
    \end{align}
    Because the right-hand side converges to $ \frac{3}{2}$ as $ n \to \infty $, 
    it is bounded and the bound is denoted by $M_2$. 

    In summary, we obtain the lemma with $ M = \max \setE*{M_1, M_2 } $. 
\end{proof}

\begin{lemma}\label{lem:the roots of quadratic equation}
    The roots of the quadratic equation, $x^2 + bx + c =0$, lie in the unit circle if and only if $ |c| < 1 $ and $ |b| < 1 + c$ hold. 
\end{lemma}

\section{Absolute stability of explicit two-step Adams method with variable step size}
\label{app:adams}

The variable step size two-step explicit Adams method can be written as 
\[ x_{n+2} - x_{n+1} = \frac{h_{n+1}}{2} \paren*{ \paren*{2 + w_{n+1} } g \paren*{ x_{n+1 } } - w_{n+1}  g \paren*{ x_{n } } } . \]
When we apply this method to Dahlquist's test equation, 
we obtain
\[ x_{n+2} - \paren*{ 1 + \frac{ \mu_{n+1} }{2} \paren*{ 2 + w_{n+1} } } x_{n+1} + \frac{\mu_{n+1}}{2} w_{n+1} x_n, \]
where $ \mu_{n+1} = \lambda h_{n+1} $. 
In order to apply \cref{prop:ltvd_stability}, we assume $ \lim_{n \to \infty } \mu_n = \mu $, which implies $ \lim_{n \to \infty } w_n = 1 $. 
Then, we obtain the characteristic polynomial
\[ z^2 - \frac{3}{2} \mu z + \frac{z}{2}, \]
which coincides with that for the \emph{fixed step size} two-step explicit Adams method. 
Therefore, \cref{prop:ltvd_stability} shows that the variable step size two-step explicit Adams method is absolutely stable when the step size $h_n$ converges to $h$ and $ \lambda h $ is in the stability region of fixed step size two-step explicit Adams method. 
Similar results hold for the general explicit/implicit Adams and BDF (Backward Differentiation Formula) methods.

\section{Derivation of the Proposed Method}\label{app:proposed}

We now solve the optimization problem~(\ref{optimization problem for proposed method}).
First, we rewrite $\widehat{c}_n$ and $\widehat{d}_n$. From the constraint~(\ref{eq:constraint1}), we have
\begin{equation}{\label{eq: widehat dn}}
    \widehat{d}_n = b_n -\frac{1}{4}(1-\widehat{c}_n + b_n)^2.
\end{equation}
We remove $\widehat{d}_n$ from the constraint~(\ref{eq:constraint2}) with (\ref{eq: widehat dn}), and obtain
\[
    \mu_n^* (1 + \mu_n^*) \widehat{c}_n^2 + (1+\mu_n^*)(b_n -1)^2 = 0.
\]
Due to the definition $\mu_n^* = -\frac{(3-3w_{n+1})^2}{(5-3w_{n+1})^2}$, $\widehat{c}_n$ can be written in
$\widehat{c}_n = \pm \frac{1-b_n}{\sqrt{-\mu_n^*}}$.
We next consider the following four cases respectively;
(i) $\widehat{c}_n = \frac{1-b_n}{\sqrt{-\mu_n^*}}$ and $b_n - \widehat{d}_n \ge b_n + \mu_n^*\widehat{d}_n$,
(ii) $\widehat{c}_n = \frac{1-b_n}{\sqrt{-\mu_n^*}}$ and $b_n - \widehat{d}_n \le b_n + \mu_n^*\widehat{d}_n$,
(iii) $\widehat{c}_n = -\frac{1-b_n}{\sqrt{-\mu_n^*}}$ and $b_n - \widehat{d}_n \ge b_n + \mu_n^*\widehat{d}_n$,
(iv) $\widehat{c}_n = -\frac{1-b_n}{\sqrt{-\mu_n^*}}$ and $b_n - \widehat{d}_n \le b_n + \mu_n^*\widehat{d}_n$.

\paragraph{case (i)}
We confirm the range of $b_n$ that satisfies $b_n - \widehat{d}_n \ge b_n + \mu_n^*\widehat{d}_n$. This inequality is identical to $0\ge (1+\mu_n^*) \widehat{d}_n$. Because $1+\mu_n^*$ is positive, it is sufficient to solve $0\ge \widehat{d}_n$ with respect to $b_n$. From the relations~\eqref{eq: widehat dn} and $\widehat{c}_n = \frac{1-b_n}{\sqrt{-\mu_n^*}}$, we obtain
\[
    \widehat{d}_n
    = b_n - \frac14\left(1 - \frac{1-b_{n}}{\sqrt{-\mu_n^*}} + b_n\right)^2
    = \frac{1}{4\mu_n^*}\left((1+\sqrt{-\mu_n^*})^2 b_n - (1 - \sqrt{-\mu_n^*})^2\right)(b_n -1).
\]
Then, $0\ge \widehat{d}_n$ holds if
\begin{equation}{\label{eq:range of bn 1}}
b_n \in \left( -\infty, \frac{(1 - \sqrt{-\mu_n^*})^2}{(1 + \sqrt{-\mu_n^*})^2} \right] \cup \left[1,\infty \right)
\end{equation} 
holds. 
Next, we minimize $b_n - \widehat{d}_n = -\frac{1}{4\mu_n^*} ((1 + \sqrt{-\mu_n^*})b_n - 1 + \sqrt{-\mu_n^*} )^2$ under (\ref{eq:range of bn 1}). It is obvious that we obtain the optimum value at $b_n =  \frac{(1 - \sqrt{-\mu_n^*})^2}{(1 + \sqrt{-\mu_n^*})^2}$, and the optimum value is $\frac{(1 - \sqrt{-\mu_n^*})^2}{(1 + \sqrt{-\mu_n^*})^2} < 1$.

\paragraph{case (ii)}
By the same argument as case (i), when $b_n$ satisfies
\begin{equation}{\label{eq:range of bn 2}}
    \frac{(1 - \sqrt{-\mu_n^*})^2}{(1 + \sqrt{-\mu_n^*})^2}\le b_n \le 1, 
\end{equation}
$b_n - \widehat{d}_n \le b_n  + \mu_n^*\widehat{d}_n$ holds. Then, we minimize $b_n + \mu_n^* \widehat{d}_n = \frac{1}{4}((1 + \sqrt{-\mu_n^*})b_n + 1 - \sqrt{-\mu_n^*})^2$ under (\ref{eq:range of bn 2}). We obtain the optimal value at $b_n =  \frac{(1 - \sqrt{-\mu_n^*})^2}{(1 + \sqrt{-\mu_n^*})^2}$, and the optimal value is $\frac{(1 - \sqrt{-\mu_n^*})^2}{(1 + \sqrt{-\mu_n^*})^2} < 1$.

\paragraph{case (iii)} 
We confirm the range of $b_n$ that satisfy $b_n - \widehat{d}_n \ge b_n + \mu_n^*\widehat{d}_n$ with $\widehat{c}_n = - \frac{1-b_n}{\sqrt{-\mu_n^*}}$.
From (\ref{eq: widehat dn}), we obtain
\[
    \widehat{d}_n 
    = b_n - \frac{1}{4}\left(1 + \frac{1-b_n}{\sqrt{-\mu_n^*}} + b_n\right)^2
    = \frac{1}{4\mu_n^*}((1 - \sqrt{-\mu_n^*})^2b_n - (1 + \sqrt{-\mu_n^*})^2 )(b_n -1).
\]
From $\mu_n^* < 0$ and $1 + \mu_n^* >0$, when 
\begin{equation}{\label{eq: range of bn 3}}
    b_n \in (-\infty,1] \cup \left[ \frac{(1 + \sqrt{-\mu_n^*})^2}{(1 - \sqrt{-\mu_n^*})^2} , \infty \right)
\end{equation}
holds, $b_n - \widehat{d}_n \ge b_n + \mu_n^*\widehat{d}_n$ is satisfied. We minimize 
\[
    b_n - \widehat{d}_n = - \frac{1}{4\mu_n^*} \left((1 - \sqrt{-\mu_n^*})b_n - 1 - \sqrt{-\mu_n^*}\right)^2
\]
under (\ref{eq: range of bn 3}). The optimum is $b_n = 1$, and the optimal value is one.

\paragraph{case (iv)}
By the same argument as case~(iii), when $b_n$ satisfies
\begin{equation}{\label{eq:range of bn 4}}
    1\le b_n \le \frac{(1 + \sqrt{-\mu_n^*})^2}{(1 - \sqrt{-\mu_n^*})^2},
\end{equation}
$b_n - \widehat{d}_n \le b_n + \mu_n^* \widehat{d}_n$ holds. We minimize $b_n + \mu_n^* \widehat{d}_n = \frac{1}{4}((1 - \sqrt{-\mu_n^*})b_n + 1 + \sqrt{-\mu_n^*})^2$ under (\ref{eq:range of bn 4}). It is obvious that the optimal value is one at $b_n = 1$.

From cases~(i-iv), we obtain $b_n = \left(\frac{1-\sqrt{-\mu_n^*}}{1 + \sqrt{-\mu_n^*}}\right)^2 = (4 -3w_{n+1})^2$, $ c_n = \frac{1 - b_n}{\sqrt{-\mu_n^*}} = (5-3w_{n+1})^2 $, and $d_n = 0$.

\end{appendix}

\end{document}